\documentclass[a4paper,10pt,leqno]{amsart}
        \title{On proofs of the Farrell-Jones Conjecture}
       \author{Bartels, A.}
       \address{Westf\"alische Wilhelms-Universit\"at M\"unster\\
               Mathematisches Institut\\
               Einsteinstr.~62,
               D-48149 M\"unster, Germany}
        \email{a.bartels@uni-muenster.de}
        \urladdr{http://www.math.uni-muenster.de/u/bartelsa} 
         \date{March 2013}

  \usepackage{calc}
  \usepackage{enumerate,amssymb}
  \usepackage[arrow,curve,matrix,tips,2cell]{xy}
    \SelectTips{eu}{10} \UseTips
    \UseAllTwocells
  \usepackage{tikz}
  \usepackage{pdfcolmk}

  \DeclareMathAlphabet{\matheurm}{U}{eur}{m}{n}

  \newcommand{\IC}{\mathbb{C}}

  \newcommand{\IN}{\mathbb{N}}

  \newcommand{\IQ}{\mathbb{Q}}
  \newcommand{\IR}{\mathbb{R}}

  \newcommand{\IZ}{\mathbb{Z}}

  \newcommand{\cala}{\mathcal{A}}
  \newcommand{\calb}{\mathcal{B}}
  \newcommand{\calc}{\mathcal{C}}

  \newcommand{\calf}{\mathcal{F}}
  
  \newcommand{\calh}{\mathcal{H}}

  \newcommand{\calo}{\mathcal{O}}

  \newcommand{\calu}{\mathcal{U}}

  \newcommand{\bfK}{{\mathbf K}}

  \newcommand{\ignore}[1]{}

  \theoremstyle{plain}
  \newtheorem{theorem}{Theorem}[section]
  \newtheorem{lemma}[theorem]{Lemma}
  
  \newtheorem{proposition}[theorem]{Proposition}
  \newtheorem{conjecture}[theorem]{Conjecture}

  \newtheorem{ABCtheorem}{Theorem}

  \theoremstyle{definition}
  \newtheorem{definition}[theorem]{Definition}
  \newtheorem*{definition*}{Definition}

  \theoremstyle{remark}
  \newtheorem{remark}[theorem]{Remark}  
  \newtheorem{convention}[theorem]{Convention}  
  \newtheorem{example}[theorem]{Example}
  \newtheorem{notation}[theorem]{Notation}

  \makeatletter\let\c@equation=\c@theorem\makeatother

   {\end{list}}

  \DeclareMathOperator{\Ab}{Ab}
  \DeclareMathOperator{\All}{All}
  
  \DeclareMathOperator{\Cyc}{Cyc}
  \DeclareMathOperator{\Fin}{Fin}
  
  \DeclareMathOperator{\id}{id}

  \DeclareMathOperator{\Idem}{Idem}
  
  \DeclareMathOperator{\pt}{pt}
  \DeclareMathOperator{\pr}{pr}
  \DeclareMathOperator{\res}{res}
  
  \DeclareMathOperator{\supp}{supp}
  \DeclareMathOperator{\Sw}{Sw}
  
  \DeclareMathOperator{\GL}{GL}
  
  \DeclareMathOperator{\Wh}{Wh}
  \DeclareMathOperator{\VCyc}{VCyc}

  \newcommand{\ab}{{\mathit{ab}}}
  
  \newcommand{\CAT}{\operatorname{CAT}}
  \newcommand{\ch}{{\mathit{ch}}}
  \newcommand{\dd}{{\partial}}
  \newcommand{\e}{{\varepsilon}}
   
  \newcommand{\fol}{{\mathit{fol}}}
  \newcommand{\FS}{{\mathit{FS}}}
  \newcommand{\hfd}{{\mathit{hfd}}}

  \newcommand{\ox}{{\otimes}}
  \newcommand{\sing}{{\mathit{sing}}}
  \newcommand{\Vcalf}{{V \hspace{-.5ex}\calf}}
  
  \newcommand{\x}{{\times}}
  
\begin{document}

  \maketitle

  \begin{abstract}
    These notes contain an introduction
    to proofs of Farrell-Jones Conjecture for 
    some groups and are based on talks given in Ohio,
    Oxford, Berlin, Shanghai, M\"unster and Oberwolfach in 2011 and 2012. 
  \end{abstract}
  

  \section*{Introduction}

  Let $R$ be a ring and $G$ be a group.
  The Farrell-Jones Conjecture~\cite{Farrell-Jones(1993a)} 
  is concerned with the
  $K$- and $L$-theory of the group ring $R[G]$.
  Roughly it says that the $K$- and $L$-theory of $R[G]$ is determined
  by the $K$- and $L$-theory of the rings $R[V]$ where $V$ 
  varies over the family of virtually cyclic subgroups of $G$ 
  and group homology.
  The conjecture is related to a number of other conjectures
  in geometric topology and $K$-theory, most prominently the
  Borel Conjecture. 
  Detailed discussions of applications and the formulation of
  this conjecture (and related conjectures) can be found 
  in~\cite{Bartels-Lueck-Reich(2008appl),Lueck(2009Hangzhou),
               Lueck(2010asph),Lueck(2011ICM),Lueck-Reich(2005)}.
  
  These notes are aimed at the reader who is already convinced 
  that the Farrell-Jones Conjecture is a worthwhile conjecture
  and is interested in recent 
  proofs~\cite{Bartels-Farrell-Lueck(2011cocomlat), 
       Bartels-Lueck(2012annals), Bartels-Lueck-Reich(2008hyper)} 
  of instances of this Conjecture.
  In these notes I  discuss aspects or special cases 
  of these proofs that I think are important and illustrating.
  The discussion is based on talks given over the last two years. 
  It will be much more informal than the actual proofs in the 
  cited papers, but I tried to provide more details than I usually do in talks.
  I took the liberty to express opinion 
  in some remarks; the reader is encouraged to disagree with me. 
  The cited results all build on the seminal work of Farrell and Jones 
  surrounding their conjecture, in particular, their introduction of the
  geodesic flow as a tool 
  in $K$- and $L$-theory~\cite{Farrell-Jones(1986a)}.
  Nevertheless, I will not assume that the reader is already 
  familiar with the methods developed by Farrell and Jones.
  
  A brief summary of these notes is as follows.
  Section~\ref{sec:statement-of-conjecture} contains a 
  brief discussion of the statement of the conjecture.
  The reader is certainly encouraged to 
  consult~\cite{Bartels-Lueck-Reich(2008appl),Lueck(2009Hangzhou),
               Lueck(2010asph),Lueck(2011ICM),Lueck-Reich(2005)} 
  for much more details, motivation and background.
  Section~\ref{sec:controlled-algebra} contains a short
  introduction to geometric modules  that is sufficient for
  these notes.
  Three axiomatic results, labeled 
  Theorems~\ref{thm:transfer-red-strict},
    ~\ref{thm:transfer-red-homotopy} and~\ref{thm:Farrell-Hsiang-groups}, 
  about the Farrell-Jones Conjecture are formulated in 
  Section~\ref{sec:thms-ABC}. 
  Checking for a group $G$ the assumptions of these results 
  is never easy.
  Nevertheless,
  the reader is encouraged to find further applications of them.
  In Section~\ref{sec:on-proof-thm-transfer-strict} an outline of 
  the proof of Theorem~\ref{thm:transfer-red-strict} is given.
  Section~\ref{sec:flow-spaces} describes the role of flows in 
  proofs of the Farrell-Jones Conjecture.
  It also contains a discussion of the flow space for
  $\CAT(0)$-groups.
  Finally, in Section~\ref{sec:Z^nxZ-is-Farrell-Hsiang} 
  an application of Theorem~\ref{thm:Farrell-Hsiang-groups}
  to some groups of the form $\IZ^n \rtimes \IZ$ is discussed.

  \subsection*{Acknowledgement}
  I had the good fortune to learn from and work with great 
  coauthors on the Farrell-Jones Conjecture; everything discussed
  here is taken from these cooperations. 
  I thank Daniel Kasprowski, Sebastian Meinert, 
  Adam Mole, Holger Reich, Mark Ullmann and 
  Christoph Winges for helpful comments on an earlier version of
  these notes.
  The work described here was supported by the 
  Sonderforschungsbereich 878 -- Groups, Geometry \& Actions.

  \section{Statement of the Farrell-Jones Conjecture}
    \label{sec:statement-of-conjecture}

  \subsection*{Classifying spaces for families}

  Let $G$ be a group. 
  A family of subgroups of $G$ is a non-empty 
  collection $\calf$ of subgroups of $G$
  that is closed under conjugation and taking subgroups.
  Examples are the family $\Fin$ of finite subgroups, the family $\Cyc$ of
  cyclic subgroups, the family of virtually cyclic subgroups $\VCyc$,
  the family $\Ab$ of abelian subgroups,
  the family $\{ 1 \}$ consisting of only the trivial subgroup and
  the family $\All$ of all subgroups. 
  If $\calf$ is a family, then the collection $\Vcalf$ of all 
  $V \subseteq G$ which contain a member of $\calf$ as a finite index
  subgroup is also a family.
  All these examples are closed under abstract isomorphism, but this
  is not part of the definition.
  If $G$ acts on a set $X$ then $\{ H \leq G \mid X^H \neq \emptyset \}$
  is a family of subgroups.

  \begin{definition} \label{def:E_calf-G}
    A $G$-$CW$-complex $E$ is called a classifying space for the family 
    $\calf$, if $E^H$ is non-empty and contractible for all
    $H \in \calf$ and empty otherwise.
  \end{definition}

  Such a $G$-$CW$-complex always exists and is unique up to 
  $G$-equivariant homotopy equivalence.
  We often say such a space  $E$ is a model for $E_\calf G$;
  less precisely we simply write $E = E_\calf G$ for such a space.
  
  \begin{example} \label{ex:E_Vcaf-G-as-full-simpl-cx}
    Let $\calf$ be a family of subgroups.
    Consider the $G$-set $S := \coprod_{F \in \calf} G/F$.
    The full simplicial complex $\Delta(S)$ spanned by $S$ (i.e., the 
    simplicial complex that contains a simplex for every non-empty 
    finite subset of $S$) carries a simplicial $G$-action.
    The isotropy groups of vertices of $\Delta(S)$ are all
    members of $\calf$, but for an arbitrary point of $\Delta(S)$ 
    the isotropy group will only contain a member of $\calf$ as a finite
    index subgroup.
    The first barycentric subdivision of $\Delta(S)$ is a $G$-$CW$-complex
    and it is not hard to see that it is a model for $E_{\Vcalf} G$.

    This construction works for any $G$-set $S$ such that
    $\calf = \{  H \leq G \mid S^H \neq \emptyset \}$.
  \end{example}

  More information about classifying spaces for families 
  can be found in~\cite{Lueck(2005s)}.
  
  \subsection*{Statement of the conjecture}
  The original formulation of the Farrell-Jones 
  Conjecture~\cite{Farrell-Jones(1993a)} used homology  
  with coefficients in stratified and twisted $\Omega$-spectra.
  We will use the elegant formulation of the conjecture 
  developed by Davis and L\"uck~\cite{Davis-Lueck(1998)}.
  Given a ring $R$ and a group $G$ Davis-L\"uck construct a 
  homology theory
  for $G$-spaces
  \begin{equation*}
    X \mapsto H^G_*( X ; \bfK_R ) 
  \end{equation*}
  with the property that $H^G_*( G/H ; \bfK_R ) = K_*(R[H])$.

  \begin{definition} \label{def:alpha_calf}
    Let $\calf$ be a  family of subgroups of $G$.
    The projection $E_\calf G \twoheadrightarrow G/G$ to the one-point
    $G$-space $G/G$ induces the $\calf$-assembly map
    \begin{equation*}
      \alpha_\calf \colon H^G_*(E_\calf G; \bfK_R) \to 
        H^G_*(G/G;\bfK_R) = K_*(R[G]).
    \end{equation*}
  \end{definition}
  
  \begin{conjecture}[Farrell-Jones Conjecture]
    \label{conj:FJ-Conj}
    For all groups $G$ and all rings $R$ the assembly map $\alpha_{\VCyc}$
    is an isomorphism.
  \end{conjecture}

  \begin{remark}
    Farrell-Jones really only conjectured this for $R = \IZ$.
    Moreover, they wrote (in 1993) that they regard this and 
    related conjectures 
    \emph{only as estimates which best fit the known data at this time}.
    It still fits all known data today.
 
    For arbitrary rings the conjecture was formulated 
    in~\cite{Bartels-Farrell-Jones-Reich(2004)}.
    The proofs discussed in this article all work for arbitrary rings
    and it seems unlikely that the conjecture holds for $R = \IZ$ 
    and all groups, but not for arbitrary rings.
  \end{remark}

  \begin{remark} \label{rem:conj-as-K=0}
    Let $\calf$ be a family of subgroups of $G$.
    If  $R$ is a ring such that
    $K_* R[F] = 0$ for all $F \in \calf$, 
    then $H_*^G(E_\calf G;\bfK_R) = 0$.

    In particular, the Farrell-Jones Conjecture predicts the following:
    if $R$ is a ring such that $K_*(R[V]) = 0$ for all $V \in \VCyc$
    then $K_* (R[G]) = 0$ for all groups $G$.
  \end{remark}
  
  \subsection*{Transitivity principle}
  The family in the Farrell-Jones Conjecture is fixed to be the 
  family of virtually cyclic groups.
  Nevertheless, it is beneficial to keep the family flexible, because 
  of the following transitivity 
  principle~\cite[A.10]{Farrell-Jones(1993a)}. 
  
  \begin{proposition}
    \label{prop:transitivity}
    Let $\calf \subseteq \calh$ be families of subgroups of $G$.
    Write $\calf \cap H$ for the family of subgroups of $H$ that
    belong to $\calf$.
    Assume that
    \begin{enumerate}
    \item \label{prop:transitivity:calg} 
      $\alpha_\calh \colon H^G_* ( E_\calh G; \bfK_R) \to K_*(R[G])$
      is an isomorphism,
    \item \label{prop:transitivity:calf} 
      $\alpha_{\calf \cap H} \colon H^H_* ( E_{\calf \cap H} H; \bfK_R) 
                   \to K_*(R[H])$ is an isomorphism for all $H \in \calh$.
    \end{enumerate}
    Then $\alpha_\calf \colon H^G_* ( E_\calf G; \bfK_R) \to K_*(R[G])$
    is an isomorphism.
  \end{proposition}

  \begin{remark} \label{rem:transitivity-for-K=0}
    The following illustrates the transitivity principle.

    Assume that $R$ is a ring such that $K_*(R[F]) = 0$ for 
    all $F \in \calf$.
    Assume moreover that the assumptions of 
    Proposition~\ref{prop:transitivity} are satisfied.
    Combining Remark~\ref{rem:conj-as-K=0} 
    with~\ref{prop:transitivity:calf}
    we conclude $K_*(R[H]) = 0$ for all $H \in \calh$.
    Then combining Remark~\ref{rem:conj-as-K=0} 
    with~\ref{prop:transitivity:calg} it follows that
    $K_*(R[G]) = 0$.
  \end{remark}

  \begin{remark} \label{rem:induction-GL_nZ}
    The transitivity principle can be used to prove the Farrell-Jones
    Conjecture for certain classes by induction.
    For example the proof
    of the Farrell-Jones Conjecture for
    $\GL_n(\IZ)$ uses an induction on 
    $n$~\cite{Bartels-Lueck-Reich-Rueping(2012)}.
    Of course the hard part is still to prove in the induction step 
    that $\alpha_{\calf_{n-1}}$ is an isomorphism for $\GL_n(\IZ)$ 
    where the family $\calf_{n-1}$  contains only groups that
    can be build from $\GL_{n-1}(\IZ)$ and poly-cyclic groups.
    The induction step uses 
    Theorem~\ref{thm:transfer-red-homotopy} from 
    Section~\ref{sec:thms-ABC}.
    See also Remark~\ref{rem:cocompact-is-important}.
  \end{remark}

  \subsection*{More general coefficients}
  Farrell and Jones also introduced a generalization of their
  conjecture now called the fibered Farrell-Jones Conjecture.
  This version of the conjecture is often not harder to prove
  than the original conjecture.
  Its advantage is that it has better inheritance properties.
  An alternative to the fibered conjecture is to allow more 
  general coefficients where the group can act on the ring.
  As $K$-theory only depends on the category of finitely generated 
  projective modules and not on the ring itself, it is natural to
  also replace the ring by an additive category.
  We briefly recall this generalization 
  from~\cite{Bartels-Reich(2007coeff)}.

  Let $\cala$ be an additive category with a $G$-action.
  There is a construction of an additive category $\cala[G]$ 
  that generalizes the twisted group ring for 
  actions of $G$ on a ring $R$. 
  (In the notation of~\cite[Def.~2.1]{Bartels-Reich(2007coeff)}
  this category is denoted as $\cala *_G G/G$; $\cala[G]$ is a more
  descriptive name for it.)
  There is also a homology theory $H^G_*(-;\bfK_\cala)$ for $G$-spaces
  such that $H^G_*(G/H;\bfK_\cala) = K_*(\cala[H])$.
  Therefore there are  assembly maps
  \begin{equation*} 
    \alpha_\calf \colon H^G_*(E_\calf G; \bfK_\cala) \to 
        H^G_*(G/G;\bfK_\cala) = K_*(\cala[G]).
  \end{equation*}

  \begin{conjecture}[Farrell-Jones Conjecture with coefficients]
    \label{Conj:FJC-coeff}
    For all groups $G$ and all additive categories $\cala$ 
    with $G$-action  the assembly map $\alpha_{\VCyc}$
    is an isomorphism.
  \end{conjecture}
  
  An advantage of this version of the conjecture is the following
  inheritance property.

  \begin{proposition}
    \label{prop:inheritance-for-extensions}
    Let $N \rightarrowtail G \twoheadrightarrow Q$ be an extension
    of groups.
    Assume that $Q$ and all preimages of virtually cyclic
    subgroups under $G \twoheadrightarrow Q$ 
    satisfies the Farrell-Jones Conjecture with 
    coefficients~\ref{Conj:FJC-coeff}.
    Then $G$ satisfies   the Farrell-Jones Conjecture with 
    coefficients~\ref{Conj:FJC-coeff}.
  \end{proposition}

  \begin{remark}
    Proposition~\ref{prop:inheritance-for-extensions} can be used to
    prove the Farrell-Jones Conjecture with coefficients for
    virtually nilpotent groups using the conjecture for 
    virtually abelian groups, 
    compare~\cite[Thm.~3.2]{Bartels-Lueck-Reich(2008appl)}. 
    
    It can also be used to reduce the conjecture for 
    virtually poly-cyclic groups to irreducible  special affine 
    groups~\cite[Sec.~4]{Bartels-Farrell-Lueck(2011cocomlat)}.
    The latter class consists of certain groups $G$ for which there
    is an exact sequence $\Delta \to G \to D$,
    where $D$ is infinite cyclic or the infinite dihedral group
    and $\Delta$ is a crystallographic group. 
  \end{remark}
  
  \begin{remark}
    For additive categories with $G$-action the consequence from
    Remark~\ref{rem:conj-as-K=0} becomes an equivalent formulation of
    the conjecture:
    A group $G$ satisfies the Farrell-Jones Conjecture with 
    coefficients~\ref{Conj:FJC-coeff} if and only if 
    for additive categories $\calb$ with $G$-action we have
    \begin{equation*}
      K_* (\calb[V]) = 0 \; \text{for all} \; V \in \VCyc
      \implies K_* (\calb[G] ) = 0.
    \end{equation*}
    (This follows from~\cite[Prop.~3.8]{Bartels-Lueck-Reich(2008hyper)} 
    because the obstruction category $\calo^G(E_\calf G;\cala)$
    is equivalent to $\calb[G]$ for some $\calb$ with
    $K_*(\calb[F]) = 0$ for all $F \in \calf$.)
    
    In particular, surjectivity implies bijectivity for
    the Farrell-Jones Conjecture with coefficients. 
  \end{remark}

  \begin{remark}
    The Farrell-Jones Conjecture~\ref{conj:FJ-Conj} should be viewed as 
    a conjecture about finitely generated groups. 
    If it holds for all finitely
    generated subgroups of a group $G$, then it holds for $G$.
    The reason for this is that the conjecture is stable under
    directed unions of groups~\cite[Thm.~7.1]{Farrell-Linnell(2003a)}.

    With coefficients the situation is even better.   
    This version of the conjecture is stable under directed colimits
    of groups~\cite[Cor.~0.8]{Bartels-Lueck(2009coeff)}.
    Consequently the Farrell-Jones Conjecture with coefficients 
    holds for all groups if and only if it holds for all
    \emph{finitely presented} groups,
    compare~\cite[Cor.~4.7]{Bartels-Echterhoff-Lueck(2008colim)}.
    It is therefore a conjecture about finitely presented groups.
  \end{remark}

  Despite the usefulness of this more general version 
  of the conjecture I will mostly ignore it in this paper to
  keep the notation a little simpler.

  \subsection*{$L$-theory}
  There is a version of the Farrell-Jones Conjecture for
  $L$-theory.
  For some applications this is very important.
  For example the Borel Conjecture asserting the rigidity    
  of closed aspherical topological manifolds follows in
  dimensions $\geq 5$ via surgery theory from the
  Farrell-Jones Conjecture in $K$- and $L$-theory.
  The $L$-theory version of the conjecture
  is very similar to the $K$-theory version.
  Everything said so far about the $K$-theory version
  also holds for the $L$-theory version.

  For some time proofs of the $L$-theoretic Farrell-Jones 
  conjecture have been considerably harder than their $K$-theoretic
  analoga.
  Geometric transfer arguments used in $L$-theory are 
  considerably more involved than their counterparts in $K$-theory.
  A change that came with considering arbitrary rings as coefficients
  in~\cite{Bartels-Farrell-Jones-Reich(2004)},
  is that transfers became more algebraic.
  It turned out~\cite{Bartels-Lueck(2012annals)} that this more
  algebraic point of view allowed for much easier $L$-theory transfers.
  (In essence, because the world of chain complexes with 
  Poincar\'e duality is  much more flexible than the world of manifolds.)
  This is elaborated  at the end of 
  Section~\ref{sec:on-proof-thm-transfer-strict}.

  I think that it is fair to say that, as far as proofs are concerned,
  there is as at the moment no significant difference between the 
  $K$-theoretic and the $L$-theoretic Farrell-Jones Conjecture.  
  For this reason $L$-theory is not discussed in much detail 
  in these notes. 

  \section{Controlled topology}
    \label{sec:controlled-algebra}

  \subsection*{The thin $h$-cobordism theorem}
  
  An \emph{$h$-cobordism} $W$ is a compact manifold whose boundary
  is a disjoint union $\dd W = \dd_0 W \amalg \dd_1 W$ 
  of closed manifolds such that the inclusions
  $\dd_0 W \to W$ and $\dd_1 W \to W$ are homotopy 
  equivalences.
  If $M = \dd_0 W$, then we say $W$ is an $h$-cobordism 
  over $M$.
  If $W$ is homeomorphic to $M \x [0,1]$, then $W$ is
  called trivial.

  \begin{definition}
    Let $M$ be a closed manifold with a metric $d$.
    Let $\e \geq 0$.  
    
    An $h$-cobordism $W$ over $M$ is said to be $\e$-controlled
    over $M$ if there exists a retraction $p \colon W \to M$ for the
    inclusion $M \to W$ and a homotopy $H \colon \id_W \to p$ such that
    for all $x \in W$ the track 
    \begin{equation*}
       \{ p (H(t,x)) \mid t \in [0,1] \} \subseteq M
    \end{equation*}
    has diameter at most $\e$.
  \end{definition}
  
  \begin{remark}
    Clearly, the trivial $h$-cobordism is $0$-controlled.
    Thus it is natural to think of being $\e$-controlled for small $\e$
    as being close to the trivial $h$-cobordism.
  \end{remark}
  
  The following theorem is due to Quinn~\cite[Thm.~2.7]{Quinn(1979a)}.
  See~\cite{Chapman(1979), Chapman-Ferry(1979), Ferry(1977)} for
  closely related results by Chapman and Ferry. 

  \begin{theorem}[Thin $h$-cobordism theorem]
    \label{thm:thin-h-cob}
    Assume $\dim M \geq 5$. 
    Fix a metric $d$ on $M$ (generating the topology of $M$).
    
    Then there is $\e > 0$ such that all $\e$-controlled 
    $h$-cobordisms over $M$ are trivial.
  \end{theorem}

  \begin{remark}
    Farrell-Jones used the thin $h$-cobordism 
    Theorem~\ref{thm:thin-h-cob} and generalizations thereof to 
    study $K_*(\IZ[G])$, $* \leq 1$. 
    For example in~\cite{Farrell-Jones(1986a)} they used the 
    geodesic flow of a negatively curved manifold $M$ to show that
    any element in $\Wh(\pi_1 M)$ could be realized by an 
    $h$-cobordism that in turn had to be trivial by an 
    application of (a generalization of) the thin $h$-cobordism
    Theorem. 
    Thus $\Wh(\pi_1 M) = 0$.
    In later papers they replaced the thin $h$-cobordism theorem  
    by controlled surgery theory and controlled pseudoisotopy theory.

    The later proofs of the Farrell-Jones Conjecture that we discuss
    here do not depend on the thin $h$-cobordism Theorem,
    controlled surgery theory or controlled pseudoisotopy theory,
    but on a more algebraic control theory that we discuss in the
    next subsection. 
  \end{remark}

  \subsection*{An algebraic analog of the thin $h$-cobordism theorem.}

  Geometric groups (later also called geometric modules) were 
  introduced by 
  Connell-Hollingsworth~
      \cite{Connell-Hollingsworth(1969geometric-groups)}.
  The theory was developed much further by, among others, Quinn
  and Pedersen and is sometimes referred to as controlled algebra.
  A very pleasant introduction to this theory is given
  in~\cite{Pedersen(2000contr-alg-K-survey)}.     

  Let $R$ be a ring and $G$ be a group.

  \begin{definition}
    Let $X$ be a free $G$-space and $p \colon X \to Z$ be a 
    $G$-map to a metric space with an isometric $G$-action.
    \begin{enumerate}
    \item A \emph{geometric $R[G]$-module over $X$} is a collection
      $(M_x)_{x \in X}$ of finitely generated free $R$-modules
      such that the following two conditions are satisfied.
      \begin{itemize}
      \item $M_x = M_{gx}$ for all $x \in X$, $g \in G$.
      \item $\{ x \in X \mid M_x \neq 0 \} = G \cdot S_0$
         for some finite subset $S_0$ of $X$.
      \end{itemize}
    \item Let $M$ and $N$ be geometric $R[G]$-modules over $X$. 
       Let $f \colon \bigoplus_{x \in X} M_x \to \bigoplus_{x \in X} N_x$
       be an $R[G]$-linear map (for the obvious $R[G]$-module structures).
       Write $f_{x'',x'}$ for the composition
       \begin{equation*}
         M_{x'} \rightarrowtail \bigoplus_{x \in X} M_x \xrightarrow{f}
           \bigoplus_{x \in X} N_x \twoheadrightarrow N_{x''}.
       \end{equation*}
       The \emph{support} of $f$ is defined as
       $\supp f := \{ (x'',x') \mid f_{x'',x'} \neq 0 \} \subseteq X \x X$.
       Let $\e \geq 0$. 
       Then $f$ is said to be \emph{$\e$-controlled over $Z$} if 
       \begin{equation*}
          d_Z(p(x''),p(x')) \leq \e \quad \text{for all} 
            \quad (x'',x') \in \supp f.
       \end{equation*}
    \item Let $M$ be a geometric $R[G]$-module over $X$. 
       Let $f \colon \bigoplus_{x \in X} M_x \to \bigoplus_{x \in X} M_x$ 
       be an $R[G]$-automorphism.
       Then $f$ is said to be an \emph{$\e$-automorphism over $Z$}
       if both $f$ and $f^{-1}$ are $\e$-controlled over $Z$. 
    \end{enumerate}
  \end{definition}

  \begin{remark}
     Geometric $R[G]$-modules over $X$ are
     finitely generated free $R[G]$-modules with an additional structure,
     namely an $G$-equivariant decomposition into $R$-modules indexed
     by points in $X$.
     This additional structure is not used to change the notion of
     morphisms which are still $R[G]$-linear maps. 
     But this structure provides an
     additional point of view for $R[G]$-linear maps:
     the set of morphisms between two geometric $R[G]$-modules 
     now carries a filtration by control.
       
     A good (and very simple) analog is the following.
     Consider finitely generated free $R$-modules.
     An additional structure one might be interested in 
     are bases for such modules.
     This additional information allows us to view
     $R$-linear maps between them as matrices.  

     Controlled algebra is really not much more than working with
     (infinite) matrices whose index set is a (metric) space.
     Nevertheless this theory is very useful and flexible. 
  \end{remark}
 
  It is a central theme in controlled topology that sufficiently
  controlled obstructions (for example Whitehead torsion) are trivial.
  Another related theme is that assembly maps can be constructed 
  as \emph{forget-control} maps.
  In this paper we will use a variation of this theme for $K_1$
  of group rings over arbitrary rings.
  Before we can state it we briefly fix some conventions
  for simplicial complexes.

  \begin{convention}
    \label{conv:calf-complex}
    Let $\calf$ be a family of subgroups of $G$.
    By a simplicial $(G,\calf)$-complex
    we shall mean a simplicial complex $E$  with a simplicial $G$-action
    whose isotropy groups $G_x = \{ g \in G \mid g \cdot x = x \}$
    belong to $\calf$ for all $x \in E$.
  \end{convention}

  \begin{convention}
    \label{conv:l1-metric}
    We will always use the $l^1$-metric on simplicial complexes.
    Let $Z^{(0)}$ be the vertex set of the simplicial complex $Z$.
    Then every element $z \in Z$ can be uniquely written as
    $z = \sum_{v \in Z^{(0)}} z_v \cdot v$ where $z_v \in [0,1]$,
    all but finitely many $z_v$ are zero and $\sum_{v \in Z^{(0)}} z_v = 1$.
    The $l^1$-metric on $Z$ is given by
    \begin{equation*}
      d^1_Z ( z,z') = \sum_{v \in V} | z_v - {z'}_v |. 
    \end{equation*}
  \end{convention}

  \begin{remark}
    If $E$ is a simplicial complex with a simplicial $G$-action such
    that the isotropy groups $G_v$ belong to $\calf$ for all vertices 
    $v \in E^{(0)}$ of $E$, then $E$ is a simplicial 
    $(G, \Vcalf)$-complex, where 
    $\Vcalf$ consists of all subgroups $H$ of $G$ that admit a subgroup of
    finite index that belongs to $\calf$.
  \end{remark}

  \begin{theorem}[Algebraic thin $h$-cobordism theorem]
    \label{thm:alg-thin-h-cob-thm}
    Given a natural number $N$, there is $\e_N > 0$ such that
    the following holds:
    Let
    \begin{enumerate}
    \item $Z$ be a simplicial $(G,\calf)$-complex of 
      dimension at most $N$,
    \item $p \colon X \to Z$ be a $G$-map, where $X$ is a free
      $G$-space,
    \item $M$ be a geometric $R[G]$-module over $X$,
    \item $f \colon M \to M$ be an $\e_N$-automorphism over $Z$
      (with respect to the $l^1$-metric on $Z$).
    \end{enumerate}
    Then the $K_1$-class $[f]$ of $f$ belongs to the image
    of the assembly map
    \begin{equation*}
      \alpha_\calf \colon H^G_1(E_\calf G; \bfK_R) \to K_1 (R[G]).
    \end{equation*}
  \end{theorem}

  \begin{remark}
      I called Theorem~\ref{thm:alg-thin-h-cob-thm} the
      algebraic thin $h$-cobordism theorem here, because it can be used 
      to prove the thin $h$-cobordism theorem.
      Very roughly, this works as follows. 
      Let $W$ be an $\e$-thin $h$-cobordism over $M$.
      Let $G = \pi_1 M = \pi_1 W$.
      The Whitehead torsion of $W$ can be constructed 
      using the  singular chain complexes of the
      universal covers $\widetilde W$ and $\widetilde M$.    
      This realizes the Whitehead torsion $\tau_W \in \Wh(G)$ of $W$ by an
      $\widetilde \e$-automorphism $f_W$ over $\widetilde M$, i.e.
      $[f_W]$ maps to $\tau_W$ under $K_1(\IZ[G]) \to \Wh(G)$.
      Moreover, $\widetilde \e$ can be explicitly bounded in terms of $\e$,
      such that $\widetilde \e \to 0$ as $\e \to 0$.
      Because $\widetilde M$ is a free $G = \pi_1 M$-space
      it follows from Theorem~\ref{thm:alg-thin-h-cob-thm} that 
      $[f_W]$ belongs to the image of the assembly map
      $\alpha \colon H_1^G(EG,\bfK_\IZ) \to K_1(\IZ[G])$.
      But $\Wh(G)$ is the cokernel of $\alpha$ and therefore
      $\tau_W = 0$.
      This reduces the thin $h$-cobordism theorem to the
      $s$-cobordism theorem.
 
      I believe that -- at least in spirit -- this outline is very close 
      to Quinn's proof in~\cite{Quinn(1979a)}.   
  \end{remark}

  \begin{remark} \label{rem:obstruction}
     If $f \colon M \to M'$ is $\e$-controlled over $Z$ and
     and $f' \colon M' \to M''$ is $\e'$-controlled over $Z$,
     then their composition $f' \circ f$ is $\e+\e'$-controlled.
     In particular, there is no category whose objects are
     geometric modules and whose morphisms are $\e$-controlled
     for fixed (small) $\e$.
     However, there are very elegant substitutes for this 
     ill-defined category.
     These are built by considering a variant of the theory
     over an open cone over $Z$ and taking a quotient category.
     In this quotient category every morphisms has for every 
     $\e > 0$ an $\e$-controlled representative.
     Pedersen-Weibel~\cite{Pedersen-Weibel(1989)} used this to 
     construct homology of a space $E$
     with coefficients in the $K$-theory spectrum 
     as the $K$-theory of an additive category.
     Similar constructions can be used to describe the assembly
     maps as forget-control 
     maps~\cite{Bartels-Farrell-Jones-Reich(2004), 
                              Carlsson-Pedersen(1995a)}. 
     This also leads to a category 
     (called the obstruction category 
     in~\cite{Bartels-Lueck-Reich(2008hyper)}),
     whose $K$-theory describes the fiber
     of the assembly map. 
     A minor drawback of these constructions is that they usually
     involve a dimension shift.

     A very simple  version of such a construction is discussed
     at the end of this section. 
     See in particular Theorem~\ref{thm:thin-h-cob-higher}.
  \end{remark}

  \begin{remark} 
      It is not hard to deduce Theorem~\ref{thm:alg-thin-h-cob-thm}
      from~\cite[Thm.~5.3]{Bartels-Lueck(2012annals)}.
      The latter result is a corresponding result for the
      obstruction category mentioned in Remark~\ref{rem:obstruction}.
      In fact this result about the obstruction 
      is stronger and can be used to prove 
      that the assembly map is an isomorphism and not just surjective, 
      see~\cite[Thm.~5.2]{Bartels-Lueck(2012annals)}. 
      I have elected to state the weaker 
      Theorem~\ref{thm:alg-thin-h-cob-thm} because it is much easier
      to state, but still grasps the heart of the matter.
      On the other hand, I think it is not at all easier to 
      prove Theorem~\ref{thm:alg-thin-h-cob-thm} than to
      prove the corresponding statement for the obstruction category.
      (The result in~\cite{Bartels-Lueck(2012annals)} deals with
      chain complexes, but this is not an essential difference.) 
  \end{remark}

  \begin{remark}   
      Results like Theorem~\ref{thm:alg-thin-h-cob-thm}  
      are very  useful to prove the Farrell-Jones
      Conjecture.
      But it is not clear to me, that it really provides the best 
      possible description of the image of the assembly map.  
      For $g \in G$ we know that $[g]$ lies in the image of
      the assembly map. 
      But its most natural representative 
      (namely the isomorphism  of $R[G]$
      given by right multiplication by $g$) is not $\e$-controlled for 
      small $\e$.  

      It may be beneficial  to find other, maybe more algebraic
      and less geometric, characterizations of the image of 
      the assembly map.
      But I do not know how to approach this.
  \end{remark}

  \begin{remark} 
      The use of the $l^1$-metric in Theorem~\ref{thm:alg-thin-h-cob-thm}
      is of no particular importance. 
      In order for $\e_N$ to only depend on $N$ and not on $Z$, one has to
      commit to some canonical metric. 
  \end{remark}

  \begin{remark}
    If $\calf$ is closed under finite index supergroups, i.e., if 
    $\calf = \Vcalf$ then there is no loss of generality in assuming that
    $Z$ is the $N$-skeleton of the model for
    $E_\calf G$ discussed in Example~\ref{ex:E_Vcaf-G-as-full-simpl-cx}.
    This holds because there is always a $G$-map
    $Z^{(0)} \to S := \coprod_{F \in \calf} G / F$
    and this map extends to a simplicial map
    $Z \to \Delta(S)^{(N)}$.
    Barycentric subdivision only changes the metric
    on the $N$-skeleton in a controlled (depending on $N$)
    way.
  \end{remark}

  \begin{remark} 
    There is also a converse to Theorem~\ref{thm:alg-thin-h-cob-thm}.
    If $a \in K_1(R[G])$ lies in the image of the
    assembly map $\alpha_\calf$ then there is some $N$ such that 
    it can for any $\e > 0$ be realized by an $\e$-automorphism 
    over an $N$-dimensional simplicial complex $Z$ with a simplicial
    $G$-action all whose isotropy groups belong to $\calf$.
    The simplicial complex can be taken to be the $N$-skeleton
    of a simplicial complex model for $E_\calf G$.

    This is a consequence of the description of the assembly map
    as a forget-control map as for example 
    in~\cite[Cor.~6.3]{Bartels-Farrell-Jones-Reich(2004)}.
  \end{remark}

  \begin{remark}
    It is not hard to extend the theory of geometric $R[G]$-modules 
    from rings to additive categories.
    In this case one considers collections $(A_x)_{x \in X}$ where
    each $A_x$ is an object from $\cala$.
    In fact~\cite[Thm.~5.3]{Bartels-Lueck(2012annals)}, which implies
    Theorem~\ref{thm:alg-thin-h-cob-thm}, is formulated 
    using additive categories as coefficients.
  \end{remark}

  \begin{remark}
    Results for $K_1$ often imply results for $K_i$, $i \leq 0$,
    using suspension rings.
    For a ring $R$, there is a suspension ring $\Sigma R$ with the
    property that $K_i (R) =  K_{i+1} (\Sigma R)$~\cite{Wagoner(1972)}.
    This construction can be arranged to be compatible with group rings: 
    $\Sigma (R[G]) = (\Sigma R)[G]$.
    A consequence of this is that for a fixed group $G$ 
    the surjectivity of 
    $\alpha_\calf \colon H^G_1(E_\calf G; \bfK_R) \to K_1 (R[G])$
    for all rings $R$ implies the surjectivity of
    $\alpha_\calf$ for all $i \leq 1$, 
    compare~\cite[Cor.~7.3]{Bartels-Farrell-Jones-Reich(2004)}.
   
    Because of this trick there is no need for a version of
    Theorem~\ref{thm:alg-thin-h-cob-thm} for $K_i$, $i \leq 0$.
  \end{remark}

  \subsection*{Higher $K$-theory}
  We end this section by a brief discussion of a version of
  Theorem~\ref{thm:alg-thin-h-cob-thm} for higher $K$-theory.
  Because there is no good concrete description of elements in
  higher $K$-theory it will use slightly more abstract language.

  Let $p_n \colon X_n \to Z_n$ be a sequence of $G$-maps where
  each $X_n$ is a free $G$-space and each $Z_n$ is a simplicial
  $(G,\calf)$-complex of dimension $N$.
  Define a category $\calc$ as follows.
  Objects of $\calc$ are sequences $(M_n)_{n \in \IN}$ where
  for each $n$, $M_n$ is a geometric $R[G]$-module over $X_n$.
  A morphism $(M_n)_{n \in \IN} \to (N_n)_{n \in \IN}$ in $\calc$
  is given by a sequence $(f_n)_{n \in \IN}$ of $R[G]$-linear maps
  $f_n \colon \bigoplus_{x \in X_n} (M_n)_x \to 
                \bigoplus_{x \in X_n} (N_n)_x$
  satisfying the following condition: there is $\alpha > 0$ such that
  for each $n$, $f_n$ is $\frac{\alpha}{n}$-controlled over $Z_n$. 
  For each $k \in \IN$,
  \begin{equation*}
    (M_n)_{n \in \IN} \mapsto \bigoplus_{x \in X_k} (M_k)_x
  \end{equation*}
  defines a functor $\pi_k$ from
  $\calc$ to the category of finitely generated free $R[G]$-modules.
  The following is a variation of~\cite[Cor.~4.3]{Bartels(2003a)}.
  It can be proven using~\cite[Thm.~7.2]{Bartels-Lueck-Reich(2008hyper)}.

  \begin{theorem} \label{thm:thin-h-cob-higher}
    Let $a \in K_* ( R[G] )$. 
    Suppose that there is $A \in K_*(\calc)$ such that 
    for all $k$ 
    \begin{equation*}
       (\pi_k)_* (A) = a.
    \end{equation*}
    Then $a$ belongs to the image of 
    $\alpha_\calf \colon H_*^G(E_\calf G; \bfK_R) \to K_*(R[G])$.
  \end{theorem}

  \section{Conditions that imply the Farrell-Jones Conjecture}
    \label{sec:thms-ABC}

  In~\cite{Bartels-Lueck(2012annals), Bartels-Lueck-Reich(2008hyper)}
  the Farrell-Jones Conjecture is proven for hyperbolic and 
  $\CAT(0)$-groups.
  Both papers take a somewhat axiomatic point of view.
  They both contain careful (and somewhat lengthy) 
  descriptions of conditions on groups that 
  imply the Farrell-Jones conjecture.
  The conditions in the two papers are closely related to each 
  other. 
  A group satisfying them is said to be~\emph{transfer reducible} 
  over a given family of subgroups in~\cite{Bartels-Lueck(2012annals)}.
  Further variants of these conditions are introduced 
  in~\cite{Bartels-Lueck-Reich-Rueping(2012),Wegner(2012CAT0)}.
  The existence of all these different versions of these conditions
  seem to me to suggest that we have not found the ideal formulation
  of them yet.
  The notion of transfer reducible groups 
  (and all its variations) can be viewed as an axiomatization 
  of the work of Farrell-Jones using the geodesic flow that began 
  with~\cite{Farrell-Jones(1986a)}. 
  Somewhat different conditions 
  -- related to work of Farrell-Hsiang~\cite{Farrell-Hsiang(1978b)} -- 
  are discussed in~\cite{Bartels-Lueck(2011method)}.

  \subsection*{Transfer reducible groups -- strict version}

  Let $R$ be a ring and $G$ be a group.

  \begin{definition}
    \label{def:transfer-space}
    An \emph{$N$-transfer space} $X$ is a compact 
    contractible metric space
    such that the following holds.

    For any $\delta > 0$ there exists a simplicial complex $K$ of 
    dimension 
    at most $N$ and continuous maps and homotopies $i \colon X \to K$, 
    $p \colon K \to X$,
    and $H \colon p \circ i \to \id_X$ such that for any $x \in X$ 
    the diameter of $\{ H(t,x) \mid t \in [0,1] \}$ is at most $\delta$. 
  \end{definition}

  \begin{example}
    Let $T$ be a locally finite simplicial tree.
    The compactification $\overline{T}$ of $T$ 
    by equivalence classes of geodesic rays is a $1$-transfer space.
  \end{example}

  \begin{ABCtheorem}
    \label{thm:transfer-red-strict}
    Suppose that $G$  is finitely generated by $S$.
    Let $\calf$ be a family of subgroups of $G$.
    Assume that there is $N \in \IN$ such that for any $\e > 0$ there are
    \begin{enumerate}
    \item \label{thm:transfer-red-strict:space}
      an $N$-transfer space $X$ equipped with
      a $G$-action, 
    \item \label{thm:transfer-red-strict:complex}
      a simplicial $(G,\calf)$-complex $E$ of dimension at most $N$, 
    \item \label{thm:transfer-red-strict:map}
      a map $f \colon X \to E$ that is $G$-equivariant up to $\e$:
      $d^1(f(s \cdot x), s \cdot f(x)) \leq \e$ for all 
      $s \in S$, $x \in X$.  
    \end{enumerate}
    Then $\alpha_{\calf} \colon H^G_*(E_\calf G;\bfK_R) \to K_*(RG)$ is an
    isomorphism. 
  \end{ABCtheorem}
    
  \begin{remark} \label{rem:hyperbolic-are-transfer-strict}
    It follows from~\cite{Bartels-Lueck-Reich(2008cover)} that
    Theorem~\ref{thm:transfer-red-strict} (with $\calf$ the family of
    virtually cyclic subgroups $\VCyc$) applies to hyperbolic groups.
  \end{remark}

  \begin{example}
    Let $G$ be a group and $K$ be a finite contractible simplicial complex
    with a simplicial $G$-action. 
    Then for the family $\calf := \calf_K$ the assembly map   
    $\alpha_{\calf} \colon H^G_*(E_\calf G;\bfK_R) \to K_*(RG)$ is an
    isomorphism. 
    This follows from Theorem~\ref{thm:transfer-red-strict}
    by setting $N := \dim K$ and 
    $X := K$, $E := K$, $f := \id_K$ (for all $\e > 0$).
    Since $K$ is finite, the group of simplicial automorphisms of $K$ is also
    finite. 
    It follow that for all $x \in K$ the isotropy group $G_x$ has finite index
    in $G$.

    The assumptions of Theorem~\ref{thm:transfer-red-strict} should be viewed 
    as a weakening of this example.
    The properties of $K$ are reflected in requirements on $X$ or on $E$
    and the existence of the map $f$ yields a strong relationship between
    $X$ and $E$. 
  \end{example}

  \begin{remark} \label{rem:willet-yu}
     Rufus Willet and Guoliang Yu pointed out that the assumption
     of Theorem~\ref{thm:transfer-red-strict} implies that the group $G$
     has finite asymptotic dimension, provided
     there is a uniform bound on the asymptotic dimension of 
     groups in $\calf$.
     The latter assumptions is of course satisfied for the family of
     virtually cyclic groups $\VCyc$. 
  \end{remark}

  \begin{remark}
    Martin Bridson pointed out that the assumptions of 
    Theorem~\ref{thm:transfer-red-strict} are formally
    very similar to the concept of  amenability for actions 
    on  compact spaces.
    The main difference is that in the latter context $E$ is replaced
    by the (infinite dimensional) space of probability measures on $G$.
  \end{remark}

  \begin{remark}  \label{rem:is-a-reformulation}
      Theorem~\ref{thm:transfer-red-strict} is a minor reformulation
      of~\cite[Thm.~1.1]{Bartels-Lueck-Reich(2008hyper)}.
      In this reference instead of the existence of $f$ the existence of
      certain covers $\calu$ of $G \x X$ are postulated.
      But the first step in the proof is to use a partition of unity to
      construct a $G$-map from $G \x X$ to the nerve $|\calu|$ of $\calu$.
      Under the assumptions formulated in 
      Theorem~\ref{thm:transfer-red-strict}
      this map is simply $(g,x) \mapsto g \cdot f( g^{-1} x)$.
      
      Avoiding the open covers makes the theorem easier to state, but
      there is no real mathematical difference.
  \end{remark}

  \begin{remark}  \label{rem:surj-on-finite-skeleton}
      The proof of Theorem~\ref{thm:transfer-red-strict} 
      in~\cite{Bartels-Lueck-Reich(2008hyper)} really shows a little bit 
      more:
      there is $M$ (depending on $N$) such that the restriction of
      $\alpha_\calf$ to
      $H^G_*(E_\calf G^{(M)};\bfK_R)$ is surjective.
      For arbitrary groups and rings with non-trivial $K$-theory in 
      infinitely   many negative degrees there will be no such $M$.
      It is reasonable to expect that groups satisfying the assumptions of
      Theorem~\ref{thm:transfer-red-strict} will also admit a finite 
      dimensional model for the space $E_\calf G$. 
  \end{remark}

  \begin{remark} \label{rem:isotropy-of-X-belongs-to-calf}
      Let $E$ be a simplicial complex of dimension $N$.  
      Let $g$ be a simplicial automorphism of $E$.
      Let $x = \sum_{v \in E^{(0)}} x_v \cdot v$ 
      be a point of $E$.
      Let $\supp x := \{ v \in E^{(0)} \mid  x_v \neq 0 \}$.
      It is a disjoint union of the sets
      \begin{eqnarray*}
        P_x & := & 
           \{ v \in \supp x \mid 
              \forall n \in \IN \; : \; g^n \in \supp x \}, \\
        D_x & := &
           \{ v \in \supp x \mid 
              \exists n \in \IN \; : \; g^n \not\in \supp x \}. 
      \end{eqnarray*}
      Observe that for $v \in D_x$, we have $d^1(x,gx) \geq x_v$.
      Assume now that $d^1(x,gx) < \frac{1}{N+1}$.
      As $\sum_v x_v = 1$ there is a vertex $v$ with 
      $v \geq \frac{1}{N+1}$.
      Such a vertex $v$ belongs then to $P_x$ and it
      follows that $\{ g^n v \mid n \in \IN \}$ is finite
      and spans a simplex of $E$ whose barycenter is fixed by $g$. 
    
      Assume now that $f \colon X \to E$ is as 
      in assumption~\ref{thm:transfer-red-strict:map} 
      of Theorem~\ref{thm:transfer-red-strict}.
      If $G_x$ is the isotropy group for $x \in X$
      (and if $G_x$ is finitely generated by $S_x$ say) then
      if $\e$ is sufficiently small it follows
      that $d^1(f(x),g f(x)) < \frac{1}{N+1}$.
      The previous observation implies then $G_x \in \calf$.

      On the other hand one can apply the Lefschetz fixed point theorem
      to the simplicial dominations to $X$ and finds for fixed $g \in G$ 
      and  each $\e > 0$ a point $x_\e \in X$ such that 
      $d(g x_\e, x_\e) \leq \e$.
      The compactness of $X$ implies that there is a fixed point 
      in $X$ for  each element of $G$.
      Altogether, it follows that $\calf$ will necessarily contain the 
      family of cyclic subgroups.   
  \end{remark}

  \begin{remark} \label{rem:hyperelemantary-fixed-points}
    Frank Quinn has shown that one can replace the family of 
    virtually cyclic groups in the Farrell-Jones Conjecture by
    the family of (possibly infinite) hyper-elementary 
    groups~\cite{Quinn(2012virtab)}.
     
    It is an interesting question whether one can 
    (maybe using Smith theory) 
    build on the  argument from 
    Remark~\ref{rem:isotropy-of-X-belongs-to-calf}
    to conclude that in order for the assumptions of 
    Theorem~\ref{thm:transfer-red-strict} to be satisfied 
    it is necessary for $\calf$ to contain this family of
    (possibly infinite) hyper-elementary groups.
  \end{remark}

  \begin{remark} 
    One can ask for which $N$-transfer spaces $X$ with a $G$-action
    it is possible to find for all $\e > 0$ a map $f \colon X \to E$
    as in assumptions~\ref{thm:transfer-red-strict:complex} 
    and~\ref{thm:transfer-red-strict:map}.

    Remark~\ref{rem:isotropy-of-X-belongs-to-calf} shows that a 
    necessary condition is
    $G_x \in \calf$ for all $x \in X$,
    but it is not clear to me that this condition is not
    sufficient.

    In light of the observation of Willet and Yu from 
    Remark~\ref{rem:willet-yu} a related 
    question is whether there is a group $G$ of 
    infinite asymptotic dimension 
    for which there is an $N$-transfer space with a $G$-action
    such that the asymptotic dimension of $G_x$, $x \in X$
    is uniformly bounded.
  \end{remark}

  \begin{remark}
    The reader is encouraged to try to check that finitely generated 
    free groups satisfy the assumptions of 
    Theorem~\ref{thm:transfer-red-strict} with respect to the family 
    of (virtually) cyclic subgroups.
    In this case one can use the compactification $\bar T$ of the
    usual tree by equivalence classes of geodesic rays 
    as the transfer space.
    I am keen to see a proof of this that is easier than
    the one coming out of~\cite{Bartels-Lueck-Reich(2008cover)} and 
    avoids flow spaces.
    Maybe a clever application of Zorn's Lemma could be useful here.

    I am not completely sure whether or not it is possible to write down 
    the maps $f \colon \bar T \to E$ in 
    assumption~\ref{thm:transfer-red-strict:map} explicitly for 
    finitely generated  free groups. 
  \end{remark}

  \subsection*{Transfer reducible groups -- homotopy version}

  Let $R$ be a ring.

  \begin{definition}
    Let $G = \langle \, S \mid R \, \rangle$ be a finitely presented group.
    A homotopy action of $G$ on a space $X$ is given by
    \begin{itemize}
    \item for all $s \in S \cup S^{-1}$ maps $\varphi_s \colon X \to X$, 
    \item for all $r = s_1 \cdot s_2 \cdots  s_l \in R$ homotopies
       $H_r \colon \varphi_{s_1} \circ  \varphi_{s_2} \circ 
             \dots \circ \varphi_{s_l} \to \id_X$
    \end{itemize}
  \end{definition}

  \begin{ABCtheorem}
    \label{thm:transfer-red-homotopy}
    Suppose that $G  = \langle \, S \mid R \, \rangle$  is
    a finitely presented group.
    Let $\calf$ be a family of subgroups of $G$.
    Assume that there is $N \in \IN$ such that for any $\e > 0$ there are
    \begin{enumerate}
    \item \label{thm:transfer-red-homotopy:space}
      an $N$-transfer space $X$ equipped with
      a homotopy $G$-action $(\varphi,H)$, 
    \item \label{thm:transfer-red-homotopy:complex}
      a simplicial $(G,\calf)$-complex $E$ of dimension at most $N$, 
    \item \label{thm:transfer-red-strict:map}
      a map $f \colon X \to E$ that is $G$-equivariant up to $\e$:
      for all $x \in X$, $s \in S \cup S^{-1}$, $r \in R$
      \begin{itemize}
      \item 
        $d^1(f( \varphi_s (x) ), s \cdot f(x)) \leq \e$,
      \item 
        $\{ H_r(t,x) \mid t \in [0,1] \}$ has diameter at most $\e$.
      \end{itemize} 
    \end{enumerate}
    Then $\alpha_{\calf} \colon H^G_i(E_\calf G;\bfK_R) \to K_i(RG)$ is an
    isomorphism for $i \leq 0$ and surjective for $i = 1$.
  \end{ABCtheorem}

  \begin{remark}
    It follows from~\cite{Bartels-Lueck(2012CAT(0)flow)} that 
    Theorem~\ref{thm:transfer-red-homotopy} applies to $\CAT(0)$-groups
    (where $\calf$ is the family of virtually cyclic groups).
    We will sketch the proof of this fact in Section~\ref{sec:flow-spaces}.

    Wegner introduced the notion of a strong homotopy action
    and proved a version of Theorem~\ref{thm:transfer-red-homotopy}
    where the conclusion is that $\alpha_\calf$ is an isomorphism in all
    degrees~\cite{Wegner(2012CAT0)}.
    This result also applies to $\CAT(0)$-groups.
  \end{remark}

  \begin{remark}
    Theorem~\ref{thm:transfer-red-homotopy} is a reformulation 
    of~\cite[Thm.~1.1]{Bartels-Lueck(2012annals)}
    (just as in Remark~\ref{rem:is-a-reformulation}).

    The assumptions of Theorem~\ref{thm:transfer-red-strict}
    feel much cleaner than the assumptions of 
    Theorem~\ref{thm:transfer-red-homotopy}.
    It would be very interesting if one could show, maybe using some 
    kind of limit that promotes a (strong) homotopy action to an actual 
    action, such  that the latter (or Wegner's variation of them) 
    do imply the former.
    
    In light of Remark~\ref{rem:willet-yu} this would imply in particular 
    that $\CAT(0)$-groups have finite asymptotic dimension and is therefore
    probably a difficult (or impossible) task. 
  \end{remark}

  \begin{remark}
    I do not know whether semi-direct products of the form 
    $\IZ^n \rtimes \IZ$ satisfy the assumptions of 
    Theorem~\ref{thm:transfer-red-homotopy}, for example if $\calf$ is 
    the family of abelian groups.
    On the other hand the Farrell-Jones Conjecture is known to hold 
    for such groups and more general for virtually poly-cyclic 
    groups~\cite{Bartels-Farrell-Lueck(2011cocomlat)}.
  \end{remark}

  \begin{remark}
    Remark~\ref{rem:surj-on-finite-skeleton} also applies to 
    Theorem~\ref{thm:transfer-red-homotopy}.
  \end{remark}

  \begin{remark}
    There is an  $L$-theory version of 
    Theorem~\ref{thm:transfer-red-homotopy}, 
    see~\cite[Thm.~1.1(ii)]{Bartels-Lueck(2012annals)}.
    There, the conclusion is that the assembly map $\alpha_{\calf_2}$
    is an isomorphism in $L$-theory where $\calf_2$ is the family of
    subgroups that contain a member of $\calf$ as a subgroup of index at
    most $2$.
    Of course $\VCyc = \VCyc_2$.
    There is no restriction on the degree $i$ in this 
    $L$-theoretic version  and so this also
    provides an $L$-theory version of 
    Theorem~\ref{thm:transfer-red-strict}.
  \end{remark}

  \subsection*{Farrell-Hsiang groups}

  \begin{definition}
    \label{def:hyperelemantary}
    A finite group $H$ is said to be \emph{hyper-elementary}
    if there exists a short exact sequence
    \begin{equation*}
       C \rightarrowtail H \twoheadrightarrow P 
    \end{equation*}
    where $C$ is a cyclic group and the order of $P$ is a prime power. 
  \end{definition}

  Quinn generalized this definition to infinite groups by allowing the 
  cyclic group to be infinite~\cite{Quinn(2012virtab)}.

  Hyper-elementary groups play a special role in $K$-theory
  because of the following result of Swan~\cite{Swan(1960a)}.
  For a group $G$ we denote by $\Sw(G)$ the Swan group of $G$.
  It can be defined as $K_0$ of the exact category of $\IZ[G]$-modules
  that are finitely generated free as $\IZ$-modules.
  This group encodes information about transfer maps in 
  algebraic $K$-theory.  

  \begin{theorem}[Swan]
    \label{thm:swan}
    For a finite group $F$ the induction maps
    combine to a surjective map
    \begin{equation*}
      \bigoplus_{H \in \calh(F)} \Sw(H) \twoheadrightarrow \Sw(F),
    \end{equation*}
   where $\calh(F)$ denotes the family of 
   hyper-elementary subgroups of $F$. 
  \end{theorem}

  Let $R$ be a ring and $G$ be a group.

  \begin{ABCtheorem}
    \label{thm:Farrell-Hsiang-groups}
    Suppose that $G$  is finitely generated by $S$.
    Assume that there is $N \in \IN$ such that for any $\e > 0$
    there are
    \begin{enumerate}
    \item a group homomorphism $\pi \colon G \to F$ where
       $F$ is finite,
    \item \label{thm:Farrell-Hsiang-groups:complex}
      a simplicial $(G,\calf)$-complex $E$ of dimension at most $N$ 
    \item \label{thm:Farrell-Hsiang-groups:map}
      a map $f \colon \coprod_{H \in \calh(F)} G / \pi^{-1}(H) \to E$
      that is $G$-equivariant up to $\e$:
      $d^1(f(s x), s \cdot f(x)) \leq \e$ for all 
      $s \in S$, $x \in \coprod_{H \in \calh(F)} G/\pi^{-1}(H)$.  
    \end{enumerate}
    Then $\alpha_{\calf} \colon H^G_*(E_\calf G;\bfK_R) \to K_*(RG)$ is an
    isomorphism. 
  \end{ABCtheorem}

  \begin{remark}
    Theorem~\ref{thm:Farrell-Hsiang-groups} is proven 
    in~\cite{Bartels-Lueck(2011method)} building on work of
    Farrell-Hsiang~\cite{Farrell-Hsiang(1978b)}.
    The main difference to Theorems~\ref{thm:transfer-red-strict} 
    and~\ref{thm:transfer-red-homotopy} is that the transfer
    space $X$ is replaced by the discrete space 
    $\coprod_{H \in \calh(F)} G/\pi^{-1}(H)$.
    It is Swan's theorem~\ref{thm:swan} that replaces the contractibility
    of $X$.

    I have no conceptual understanding of
    Swan's theorem.
    For this reason  Theorem~\ref{thm:Farrell-Hsiang-groups} is to me 
    not as conceptually satisfying as 
    Theorem~\ref{thm:transfer-red-strict}.
    Moreover, I expect that a version of 
    Theorem~\ref{thm:Farrell-Hsiang-groups} for Waldhausen's $A$-theory
    will need a larger family than the family
    of hyper-elementary subgroups.
  \end{remark}

  \begin{remark}
    Groups satisfying the assumption of 
    Theorem~\ref{thm:Farrell-Hsiang-groups} are called
    \emph{Farrell-Hsiang groups with respect to $\calf$}
    in~\cite{Bartels-Lueck(2011method)}. 
  \end{remark}

  \begin{remark} \label{rem:poly-cyclic}
    Theorem~\ref{thm:Farrell-Hsiang-groups} can be used to
    prove the Farrell-Jones Conjecture for virtually  poly-cyclic 
    groups~\cite[Sec.~3 and 4]{Bartels-Farrell-Lueck(2011cocomlat)}. 
    We will discuss some semi-direct products of the form $\IZ^n \rtimes \IZ$
    in Section~\ref{sec:Z^nxZ-is-Farrell-Hsiang}.
  \end{remark}

  \begin{remark}
    Remark~\ref{rem:surj-on-finite-skeleton} also applies to 
    Theorem~\ref{thm:Farrell-Hsiang-groups}.
  \end{remark}  

  \begin{remark}
    Theorem~\ref{thm:Farrell-Hsiang-groups} holds without change
    in $L$-theory as well~\cite{Bartels-Lueck(2011method)}.
  \end{remark}

  \begin{remark}
    It would be good to find a natural common weakening of the
    assumptions in 
    Theorems~\ref{thm:transfer-red-strict},\ref{thm:transfer-red-homotopy}
    and~\ref{thm:Farrell-Hsiang-groups}
    that still implies the Farrell-Jones Conjecture.
    Ideally such a formulation should have similar
    inheritance properties as the Farrell-Jones Conjecture,
    see Propositions~\ref{prop:transitivity} 
    and~\ref{prop:inheritance-for-extensions}.
  \end{remark}

  \subsection*{Injectivity}
    It is interesting to note that injectivity of the assembly map 
    $\alpha_{\{1\}}$ or $\alpha_{\Fin}$ is known for seemingly much bigger 
    classes of groups, than the class of groups known to satisfy
    the Farrell-Jones Conjecture.
    Rational injectivity of the $L$-theoretic assembly map $\alpha_{\{1\}}$
    is of particular interest, as it implies Novikov's conjecture on
    the homotopy invariance of higher signatures.
    Yu~\cite{Yu(2000)} proved the Novikov conjecture for all
    groups admitting a uniform embedding into a Hilbert-space.
    This class of groups contains all groups of 
    finite asymptotic dimension.
    Integral injectivity of the assembly map $\alpha_{\{1\}}$ for 
    $K$- and $L$-theory is known
    for all groups that admit a finite $CW$-complex as a model for 
    $BG$ and
    are of finite decomposition complexity~\cite{
     Guentner-Tessera-Yu(2012geom-complx-top-rigid),
          Ramras-Tessear-Yu(2011)}.
    The latter property is a generalization of finite asymptotic 
    dimension.
    Rational injectivity  of the $K$-theoretic assembly map 
    $\alpha_{\{1\}}$ for the ring $\IZ$ is proven by
    B\"okstedt-Hsiang-Madsen~\cite{Boekstedt-Hsiang-Madsen(1993)}
    for all groups $G$ satisfying the  following homological 
    finiteness condition: for all $n$ the rational group-homology 
    $H_*(G;\IQ)$ is finite dimensional.

  \section{On the Proof of Theorem~\ref{thm:transfer-red-strict}}
    \label{sec:on-proof-thm-transfer-strict}

  Using the results from controlled topology discussed in 
  Section~\ref{sec:controlled-algebra} we will outline a
  proof of the surjectivity of
  \begin{equation*}
    \alpha_\calf \colon H_1^G(E_\calf G;\bfK_R) \to K_1(R[G])
  \end{equation*}
  under the assumptions of Theorem~\ref{thm:transfer-red-strict}.

  \subsection*{Step 1: preparations}
  Let $G$ be a finitely generated group and 
  $\calf$ be a family of subgroups of $G$.
  Let $N \in \IN$ be the number appearing in 
  Theorem~\ref{thm:transfer-red-strict}.
  Let $a \in K_1( R[G] )$.
  Then $a = [ \psi ]$ where $\psi \colon R[G]^n \to R[G]^n$
  is an $R[G]$-right linear automorphism.
  We write $R[G]^n = \IZ[G] \ox_\IZ R^n$.  
  There is a finite subset $T \subseteq G$ and there are
  $R$-linear maps $\psi_{g} \colon  R^{n} \to R^{n}$,
  ${\psi^{-1}}_{g} \colon R^n \to R^n$, $g \in T$  such that
  \begin{equation*}
    \psi (h \ox v)  =  \sum_{g \in T}  hg^{-1} \ox \psi_{g}(v)  
   \quad \text{and} \quad
    \psi^{-1} (h \ox v)  =  \sum_{g \in T} hg^{-1} \ox  {\psi^{-1}}_{g}(v). 
  \end{equation*}
  Because of Theorem~\ref{thm:alg-thin-h-cob-thm} it suffices
  to find
  \begin{itemize}
  \item a $G$-space $Y$,
  \item a $(G,\calf)$-complex $E$ of dimension at most $N$, 
  \item a $G$-map $Y \to E$,
  \item a geometric $R[G]$-module $M$ over $Y$,
  \item an $\e_N$-automorphism over $E$, $\varphi \colon M \to M$,
  \end{itemize}
  such that $[\varphi] = a \in K_1(R[G])$. 
  Here $\e_N$ is the number depending on $N$, whose existence is
  asserted in Theorem~\ref{thm:alg-thin-h-cob-thm}.

  Let $L$ be a (large) number.
  We will later specify $L$; it will only depend on $N$.
  From the assumption of Theorem~\ref{thm:transfer-red-strict}
  we easily deduce that there are
  \begin{enumerate}
    \item \label{space}
      an $N$-transfer space $X$ equipped with
      a $G$-action, 
    \item \label{complex}
      a simplicial $(G,\calf)$-complex $E$ of dimension at most $N$, 
    \item \label{map}
      a map $f \colon X \to E$ such that
      $d^1(f(g \cdot x), g \cdot f(x)) \leq \e_N / 2$ for all 
      $x \in X$ and all $g \in G$ that can be written as
      $g = g_1 \dots g_L$ with $g_1,\dots,g_L \in T$.  
  \end{enumerate}
  By compactness of $X$ there is $\delta_0 > 0$ such that
  $d^1(f(x),f(x')) \leq \e_N / 2$ for all
  $x, x' \in X$ with $d(x,x') \leq L \delta_0$.  
  We will use $Y := G \x X$ with the $G$-action defined by
  $g \cdot (h,x) := (gh,x)$.
  We will also use the $G$-map $G \x X \to E$, $(g,x) \mapsto g f(x)$.
  The action of $G$ on $X$ will be used later.

  \subsection*{Step 2: a chain complex over $X$.}
  To simplify the discussion let us assume that for $X$ 
  the maps $i$ and $p$ appearing in Definition~\ref{def:transfer-space}
  can be arranged to be $\delta$-homotopy equivalences.
  This means that in addition to $H$ there is also a 
  homotopy $H' \colon i \circ p \to \id_K$ such that for any $y \in K$ 
  the diameter of $\{ H'(t,y) \mid t \in [0,1] \}$ 
  with respect to the $l^1$-metric on $K$ is at most $\delta$. 

  Let $C_*$ be the simplicial chain complex of the $l$-fold
  simplicial subdivision of $K$.
  Using $p \colon K \to X$ we can view $C_*$ as a 
  chain complex of geometric $\IZ$-modules over $X$.
  If we choose $l$ sufficiently large, then we can arrange that
  the boundary maps $\dd^{C_*}$ of $C_*$ are 
  $\delta_0$-controlled over $X$.
  Moreover, using the action of $G$ on $X$ and a 
  $\delta$-homotopy equivalence between $K$ and $X$ 
  (for $0 < \delta \ll \delta_0$) 
  and enlarging $l$ we can produce
  chain maps $\varphi_g \colon C_* \to C_*$, $g \in G$, chain homotopies 
  $H_{g,h} \colon \varphi_{g} \circ \varphi_{h} \to \varphi_{gh}$
  satisfying the following control conditions 
  \begin{itemize}
  \item if $g \in T$ and $(x',x) \in \supp \varphi_g$
    then $d(x', g x) \leq \delta_0$
    (recall that we view $C_*$ as  a chain complex over $X$),
  \item if $g,h \in T$ and $(x',x) \in \supp H_{g,h}$
    then $d(x', gh x) \leq \delta_0$.
  \end{itemize}
  
  \begin{remark}
    If we drop the additional assumption on $X$ (i.e., if we
    no longer assume the existence of the homotopy $H'$), then
    it is only possible to construct the chain complex $C_*$
    in the idempotent completion of geometric $\IZ$-modules
    over $X$.
    This is a technical but -- I think -- not very important point.
  \end{remark}
  
  \begin{remark}
    A construction very similar to this step 2 is carried out in great 
    detail in~\cite[Sec.~8]{Bartels-Lueck(2012annals)}. 
  \end{remark}

  \subsection*{Step 3: transfer to a chain homotopy equivalence}
  Recall our automorphism $\psi$ of $R[G]^n = \IZ[G] \ox_\IZ R^n$.
  We will now replace the $R$-module $R^n$ by the $R$-chain complex
  $C_* \ox_\IZ R^n$ to obtain the chain complex 
  $D_* := \IZ[G] \ox_\IZ C_* \ox_\IZ R^n$.
  As $C_*$ is a chain complex of geometric $\IZ$-modules over $X$,
  $D_*$ is naturally a geometric $R[G]$-module over $G \x X$.
  Here $(D_*)_{h,x} = \{ h \ox w \ox v \mid v \in R^n, w \in (C_*)_x \}$
  for $h \in G$, $x \in X$.
  Recall that we use the left action defined by 
  $g \cdot (h,x) = (gh, x)$  on $G \x X$. 
  We can now use the data from Step~2 to transfer $\psi$ to a 
  chain homotopy equivalence 
  $\Psi = \sum_{g \in T} g \ox \varphi_g \ox \psi_g \colon D_* \to D_*$. 
  Similarly, there is a chain homotopy inverse $\Psi'$ for $\Psi$
  and there are chain homotopies 
  $\calh \colon \Psi \circ \Psi' \to \id_{D_*}$ and
  $\calh' \colon  \Psi'  \circ \Psi  \to \id_{D_*}$.
  In more explicit formulas these are defined by
  \begin{eqnarray*}
    \Psi ( h \ox w \ox v ) & =  &
       \sum_{g \in T} hg^{-1}  \ox \varphi_{g} (w) \ox \psi_{g}(v), \\ 
    \Psi' ( h \ox w \ox v ) & = & 
       \sum_{g \in T} hg^{-1}  \ox \varphi_{g} (w) \ox {\psi^{-1}}_{g}(v), \\
    \calh ( h \ox w \ox v ) & = &
       \sum_{g,g' \in T} h (gg')^{-1}  
               \ox H_{g,g'} (w) \ox \psi_{g} \circ {\psi^{-1}}_{g'}(v), \\
    \calh' ( h \ox w \ox v ) & = &
       \sum_{g,g' \in T} h (gg')^{-1}  
               \ox H_{g,g'} (w) \ox {\psi^{-1}}_{g} \circ \psi_{g'}(v), \\
  \end{eqnarray*}
  for $h \in G$, $w \in C_*$, $v \in R^n$.
   
  \subsection*{Digression on torsion}
  Let $S$ be a ring.
  If $\Phi$ is a self-homotopy equivalence of a bounded chain complex $D_*$
  of finitely generated free $S$-modules then its \emph{self-torsion}
  $\tau(\Phi) \in K_1(S)$ is the $K$-theory class of an 
  automorphism ${\tilde \tau} (\Phi)$ of $\bigoplus_{n \in \IZ} D_n$.
  There is an explicit formula for $\tilde \tau (\Phi)$ that involves
  the boundary map of $D_*$, $\Phi$, a chain homotopy inverse 
  $\Phi'$ for $\Phi$ and chain homotopies 
  $\Phi \circ  \Phi' \to \id_{D_*}$,
  $\Phi' \circ \Phi \to \id_{D_*}$.
  The ingredients for such a formula can be found for example 
  in~\cite[Sec.~12.1]{Bartels-Farrell-Jones-Reich(2004)}. 
  A key property is that given a commutative diagram
  \begin{equation*}
   \xymatrix{ D_* \ar[rr]^{\Phi_1} \ar[d]^{q} & & 
              D_* \ar[d]^{q} \\
              D'_* \ar[rr]^{\Phi_2} & &
              D'_* 
            }
  \end{equation*}
  where $\Phi_1$, $\Phi_2$ and $q$ are chain homotopy equivalences
  one has $\tau(\Phi_1) = \tau(\Phi_2) \in K_1(S)$. 
  
  \begin{remark}
    It is possible to formulate Theorem~\ref{thm:alg-thin-h-cob-thm}
    directly for self-chain homotopy equivalences of chain complexes
    of geometric modules of bounded dimension.
    Then the discussion of torsion can be avoided here.
    This is the point of view taken 
    in~\cite[Thm.~5.3]{Bartels-Lueck(2012annals)}. 
  \end{remark}

  \subsection*{Step 4:  $\tau( \Psi ) = a$.}
  Because $X$ is contractible, the augmentation map $C_* \to \IZ$ 
  induces a homotopy  equivalence 
  \begin{equation*}
    q \colon D_* = \IZ[G] \ox_\IZ C_* \ox_\IZ R^n \to 
         \IZ[G] \ox_\IZ \IZ \ox_\IZ R^n = \IZ[G] \ox_\IZ R^n.
  \end{equation*}
  Moreover, $q \circ \Psi = \psi \circ q$.
  It follows that 
  \begin{equation*}
    a = [ \psi ] = \tau ( \psi ) = \tau (\Psi)
      = [ \tilde \tau ( \Psi ) ]
  \end{equation*}
  
  \subsection*{Step 5: control of $\tilde \tau (  \Psi ) $.}
  In order to understand the support of $\tilde \tau (\Psi ) $
  we first need to understand the support of its building blocks.
  If $((h',x'),(h,x)) \in (G \x X)^2$ belongs to the support of
  $\dd^{D_*}$, then $h' = h$ and $d(x',x) \leq \delta_0$.
  If $((h',x'),(h,x))$ belongs to the support of  $\Psi$
  or of its homotopy inverse $\Psi'$, then there is $g \in T$ such that
  $h' = hg^{-1}$ and $d(x',gx) \leq \delta_0$.
  If $((h',x'),(h,x))$ belongs to the support of the 
  chain homotopy $\calh$ or $\calh'$ 
  then there are $g, g' \in T$ such that
  $h' = h (gg')^{-1}$ and $d(x',gg'x) \leq \delta_0$.
  From the explicit formula for $\tilde \tau ( \Psi )$
  one can then read off that there is a number $K$, depending only
  on the dimension of $D_*$ (which is in our case bounded by $N$),
  such that the support of $\tilde \tau ( \Psi )$
  satisfies the following condition:
  if $((h',x'),(h,x)) \in \supp \tilde \tau ( \Psi  )$
  then there are $g_1,\dots,g_K \in T$ such that
  \begin{equation*}
     h' = h(g_1 \dots g_K)^{-1}  \quad \quad \text{and} \quad 
     d(x',g_1 \dots g_K x) \leq K\delta_0.
  \end{equation*}
  Note that we specified $K$ in this step; 
  note also that $K$ does only depend on $N$.
  
  \begin{remark}
    The actual value of $K$ is of course not important.
    It is not very large; for example $K := 10 N$ works -- I think.
  \end{remark}

  \subsection*{Step 6: applying $f$.}
  Using the map $f \colon X \to E$ we define the $G$-map 
  $F \colon G \x X \to E$ by $F(h,x) := h f(x)$.
  Combining the estimates from the end of step~2 and 
  the analysis of $\supp (\tilde \tau ( \Psi ))$
  it is not hard to see that $\tilde \tau ( \Psi )$
  is an $\e_N$-automorphism over $E$ (with respect to $F$).
  
  This finishes the discussion of the surjectivity of 
  $ \alpha_\calf \colon H_1^G(E_\calf G;\bfK_R) \to K_1(R[G])$
  under the assumptions of Theorem~\ref{thm:transfer-red-strict}.
  Surjectivity of this map under the assumptions of
  Theorem~\ref{thm:transfer-red-homotopy} follows from a very similar
  argument; mostly step~2 is slightly more complicated.
  For Theorem~\ref{thm:Farrell-Hsiang-groups} the transfer
  can no longer be constructed using a chain complex associated to a 
  space; instead Swan's Theorem~\ref{thm:swan} is used to construct
  a transfer. 
  Otherwise the proof is again very similar.

  \subsection*{$L$-theory transfer}
  The proof of the $L$-theory version of 
  Theorems~\ref{thm:transfer-red-strict}  
  and~\ref{thm:transfer-red-homotopy}
  follows the same outline.
  Now elements in $L$-theory are given by quadratic forms.
  The analog of chain homotopy self-equivalences in $L$-theory
  are ultra-quadratic Poincar\'e complexes~\cite{Ranicki(1992)}.
  These are chain complex versions of quadratic forms.
  The main difference is that to construct a transfer it
  is no longer sufficient to have just the chain complex $C$,
  in addition we need a symmetric structure on this
  chain complex.
  Moreover, this symmetric structure needs to be controlled
  (just as the boundary map $\dd$ is controlled).
  While there may be no such symmetric structure on
  $C$, there is a symmetric structure on the product of $C$ with its dual
  $D := C \ox C^{-*}$.
  This symmetric structure is given (up to signs) by
  $\langle a \ox \alpha, b \ox \beta \rangle = \alpha(b)\beta(a)$
  and turns out to be suitably controlled.
  This is the only significant change from the proof in $K$-theory
  to the proof in $L$-theory.
  
  \subsection*{Transfer for higher $K$-theory}
  We end this section with a very informal discussion of 
  one aspect of the proof of Theorem~\ref{thm:transfer-red-strict} 
  for higher $K$-theory.
  Again, we focus on surjectivity.
  In this case we use Theorem~\ref{thm:thin-h-cob-higher}
  in place of Theorem~\ref{thm:alg-thin-h-cob-thm}.
  Thus we need to produce an element in $K_*(\calc)$.
  Recall that objects of $\calc$ are sequences $(M_\nu)_{\nu \in \IN}$ 
  of geometric $R[G]$-modules and that morphisms are sequences of 
  $R[G]$-linear maps that become more controlled with $\nu \to \infty$.
  The general idea is to apply the transfer from Step~3 to each $\nu$
  to produce a functor from $R[G]$-modules to $\calc$.
  The problem is, however, 
  that the construction from Step~3 is not functorial.
  The reason for this in turn is that the group $G$ only acts up to
  homotopy on the chain complex $C_*$.
  The remedy for this failure is to use the singular chain complex 
  of $C_*^\sing(X)$ in place of $C_*$.
  It is no longer finite, but it is homotopy finite, 
  which is finite enough.
  For the control consideration from Step~5 it was important, that
  the boundary map of $C_*$ is $\delta_0$-controlled.
  This is no longer true for $C^\sing_*(X)$.
  One might be tempted to use the subcomplex $C^{\sing,\delta_0}(X)$
  spanned by singular simplices in $X$ of diameter $\leq \delta_0$.
  However,  the action of $G$ on $X$ is not 
  isometric and therefore there is no $G$-action 
  on $C^{\sing,\delta_0}(X)$.
  Finally, the solution is to use $C_*^\sing(X)$ together with its
  filtration by the subcomplexes $(C_*^{\sing,\delta}(X))_{\delta > 0}$.
  Using this idea it is possible to construct a transfer functor
  from the category of $R[G]$-modules to a category
  $\widetilde{\ch}_{\hfd} \calc$.
  The latter is a formal enlargement of the Waldhausen category
  ${\ch}_{\hfd} \calc$ of homotopy finitely 
  dominated chain complexes over the category 
  $\calc$~\cite[Appendix]{Bartels-Reich(2005jams)}.
  Both the higher $K$-theory of $\ch_{\hfd} \calc$
  and of $\widetilde{\ch}_{\hfd} \calc$ coincide with
  the higher $K$-theory of $\calc$.
  Similar constructions are used 
  in~\cite{Bartels-Lueck-Reich(2008hyper), Wegner(2012CAT0)}.

  \section{Flow spaces}
    \label{sec:flow-spaces}
  
  \begin{convention}
    A \emph{$\CAT(0)$-group} is
    a group that admits a cocompact, proper 
    and isometric action on a finite dimensional $\CAT(0)$-space.    
  \end{convention}

  The goal of this section is to outline the proof 
  of  the fact~\cite{Bartels-Lueck(2012CAT(0)flow)} that 
  $\CAT(0)$-groups satisfy the assumptions of 
  Theorem~\ref{thm:transfer-red-homotopy}.
  Note that $\CAT(0)$-groups are finitely 
  presentable~\cite[Thm.~III.$\Gamma$.1.1(1), p.439]
                                 {Bridson-Haefliger(1999)}.

  \begin{proposition} \label{prop:CAT(0)-satisfy-B}
    Let $G$ be a $\CAT(0)$-group.
    Exhibit $G$ as a finitely presented group $\langle S \mid R \rangle$. 
    Then there is $N \in \IN$ such that for any $\e > 0$ there are
    \begin{enumerate}
    \item \label{thm:transfer-red-homotopy:space}
      an $N$-transfer space $X$ equipped with
      a homotopy $G$-action $(\varphi,H)$, 
    \item \label{thm:transfer-red-homotopy:complex}
      a simplicial $(G,\VCyc)$-complex $E$ of dimension at most $N${,} 
    \item \label{thm:transfer-red-strict:map}
      a map $f \colon X \to E$ that is $G$-equivariant up to $\e$:
      for all $x \in X$, $s \in S \cup S^{-1}$, $r \in R$
      \begin{itemize}
      \item 
        $d^1(f( \varphi_s (x) ), s \cdot f(x)) \leq \e$,
      \item 
        $\{ H_r(t,x) \mid t \in [0,1] \}$ has diameter at most $\e$.
      \end{itemize} 
    \end{enumerate}
  \end{proposition}
  
  \subsection*{An $(\alpha,\e)$-version of the assumptions
         of Theorem~\ref{thm:transfer-red-homotopy} }

  Let $G$ be a group.

  \begin{definition} \label{def:N-flow-space}
    An $N$-flow space $\FS$ for $G$ is a metric space
    with a continuous flow $\phi \colon \FS \x \IR \to \FS$ and
    an isometric proper action of $G$ such that
    \begin{enumerate}
    \item the flow is $G$-equivariant: $\phi_t(gx) = g \phi_t(x)$
       for all $x \in X$, $t \in \IR$ and $g \in G$;
    \item $\FS \setminus 
            \{ x \mid \phi_t(x) = x \; \text{for all} \; t \in \IR \}$
       is locally connected and has covering dimension at most $N$. 
    \end{enumerate}
  \end{definition}

  \begin{notation}
    Let $\alpha, \e \geq 0$. For $x,y \in \FS$ we write
    \begin{equation*}
      d_\FS^\fol(x,y) \leq (\alpha,\e)
    \end{equation*}
    if there is $t \in [-\alpha,\alpha]$ such that 
    $d ( \phi_t (x), y) \leq \e$. 
  \end{notation}
  
  Of course, $\e$ will usually be a small number while $\alpha$
  will often be much larger.

  \begin{proposition} \label{prop:(alpha-e)-B-for-CAT(0)}
    Let $G$ be a $\CAT(0)$-group.
    Exhibit $G$ as a finitely presented group $\langle S \mid R \rangle$. 
    Then there exists $N \in \IN$ and a cocompact $N$-flow space
    for $G$ and $\alpha > 0$ such that for all $\e > 0$
    there are
    \begin{enumerate}
    \item \label{prop:(alpha-e)-B-==>-B:space}
      an $N$-transfer space $X$ equipped with
      a homotopy $G$-action $(\varphi,H)$, 
    \item \label{prop:(alpha-e)-B-==>-B:map}
      a map $f \colon X \to \FS$ that is $G$-equivariant up to 
      $(\alpha,\e)$:
      for all $x \in \FS$, $s \in S \cup S^{-1}$, $r \in R$, $t \in [0,1]$
      \begin{itemize}
      \item 
        $d_\FS^\fol(f( \varphi_s (x) ), s \cdot f(x)) \leq (\alpha,\e)$,
      \item  $d_\FS^\fol ( f(H_r(t,x)), f(x)) \leq (\alpha,\e)$.
      \end{itemize} 
    \end{enumerate}
  \end{proposition}

  The proof of Proposition~\ref{prop:(alpha-e)-B-for-CAT(0)}
  will be discussed in a later subsection.
  The key ingredient that allows to deduce  
  Proposition~\ref{prop:CAT(0)-satisfy-B}
  from  Proposition~\ref{prop:(alpha-e)-B-for-CAT(0)} 
  are the long and thin covers
  for flow spaces from~\cite{Bartels-Lueck-Reich(2008cover)},
  that in turn generalize the long and thin cell structures 
  of Farrell-Jones~\cite[Sec.~7]{Farrell-Jones(1986a)}.
  
  \begin{definition}
    Let $R > 0$. 
    A collection $\calu$ of open subsets of $\FS$ is said to be 
    an \emph{$R$-long cover} of $A \subseteq \FS$ if for all
    $x \in A$ there is $U \in \calu$ such that
    \begin{equation*}
      \phi_{[-R,R]}(x) := \{ \phi_t(x) \mid t \in [-R,R] \} \subseteq U.
    \end{equation*}
  \end{definition}

  \begin{notation}(Periodic orbits)
    Let $\gamma > 0$.
    Write $\FS_{\leq \gamma}$ for the subset of $\FS$ consisting
    of all points $x$ for which there are $0 < \tau \leq \gamma$ 
    and $g \in G$ such that $\phi_\tau(x) = gx$. 
  \end{notation}

  \begin{theorem}[Existence of long thin covers]
    \label{thm:ex-long-thin-cover}
    Let $\FS$ be a cocompact $N$-flow space for $G$.
    Then there is $\tilde N$ such that for all
    $R > 0$ there exists $\gamma > 0$ and a collection $\calu$
    of open subsets of $\FS$ such that 
    \begin{enumerate}
    \item $\dim \calu \leq \tilde N$: any point of $\FS$
      is contained in at most $\tilde N + 1$ members of $\calu$,
    \item $\calu$ is an $R$-long cover of 
      $\FS \setminus \FS_{\leq \gamma}$,
    \item $\calu$ is $G$-invariant: for $g \in G$, $U \in \calu$
      we have $g(U) \in \calu$,
    \item \label{thm:ex-long-thin-cover:fin} $\calu$ has finite isotropy:
      for all $U \in \calu$ the group 
      $G_U := \{ g \in G \mid g(U) = U \}$ is finite.
    \end{enumerate}
  \end{theorem}

  \begin{example} \label{ex:Z-on-R}
    Let $G := \IZ$.
    Consider $\FS := \IR$ with the usual $\IZ$-action and the flow
    defined by $\phi_t(x) := x+t$.
    If $\calu_R$ is an $R$-long $\IZ$-invariant cover of $\IR$ of
    finite isotropy then
    the dimension of $\calu_R$  grows linearly with $R$.
     
    Theorem~\ref{thm:ex-long-thin-cover} states that this is 
    the only obstruction to the existence of uniformly finite dimensional 
    arbitrary long $G$-invariant covers of $\FS$ of finite isotropy. 
  \end{example}

  \begin{remark} \label{rem:proof-of-ex-long-thin}
    Theorem~\ref{thm:ex-long-thin-cover} is more or 
    less~\cite[Thm.~1.4]{Bartels-Lueck-Reich(2008cover)},
    see also~\cite[Thm.~5.6]{Bartels-Lueck(2012CAT(0)flow)}.  
    The proof depends only on fairly elementary constructions,
    but is nevertheless very long.    
    (It would be nice to simplify this proof -- but I do not
    know where to begin.)

    In these references in addition an upper bound for the order
    of finite subgroups of $G$ is assumed.
    This assumption is removed in recent (and as of yet unpublished) 
    work of Adam Mole and Henrik R\"uping.
  \end{remark}

  \begin{remark} \label{rem:extend-to-all-of-FS}
    For the flow spaces, that have been relevant for the Farrell-Jones
    conjecture so far, it is possible to extend the
    cover $\calu$ from $\FS \setminus \FS_{\leq \gamma}$ to
    all of $\FS$.
    The only price one has to pay for this extension is that in 
    assertion~\ref{thm:ex-long-thin-cover:fin} one has to allow
    virtually cyclic groups instead of only finite groups. 
    Note that with this change Example~\ref{ex:Z-on-R} is no longer
    problematic; we can simply set $\calu_R := \{ \IR \}$. 

    It is really at this point where the family of virtually cyclic
    subgroups plays a special role and appears in proofs of
    the Farrell-Jones Conjecture. 
  \end{remark}

  \begin{remark}
     In the case of $\CAT(0)$ groups the extension of the cover from
     $\FS \setminus \FS_\gamma$ to $\FS$ is really the most
     technical part of the arguments 
     in~\cite{Bartels-Lueck(2012CAT(0)flow)}.
     
     It would be more satisfying to have a result that provides this
     extension (after allowing virtually cyclic groups) 
     for general cocompact flow spaces.
  \end{remark}
  
  \begin{remark}
     One may think of Theorem~\ref{thm:ex-long-thin-cover} as
     a (as of now quite difficult!) parametrized version of the 
     very easy fact  that $\IZ$ has finite  asymptotic dimension.
  \end{remark}

  \subsection*{Sketch of proof for 
          Proposition~\ref{prop:CAT(0)-satisfy-B}
          using Proposition~\ref{prop:(alpha-e)-B-for-CAT(0)}}
  
  The idea is  easy.
  We produce a map $F \colon \FS \to E$ that
  is suitably contracting along the flow lines of $\phi$.
  Then we can compose $f \colon X \to \FS$ from
  Proposition~\ref{prop:(alpha-e)-B-for-CAT(0)}
  with $F$ to produce the required map $F \circ f \colon X \to E$.  
 
  Let $G$ be a $\CAT(0)$-group. 
  Let $\e > 0$ be given.
  Let $\FS$ be the cocompact $N$-flow space for $G$ from 
  Proposition~\ref{prop:(alpha-e)-B-for-CAT(0)}. 
  As discussed in Remark~\ref{rem:extend-to-all-of-FS} there
  is $\tilde N$ such that for all 
  $R > 0$ there exists a collection $\calu$
  of open subsets of $\FS$ such that 
  \begin{enumerate}
    \item $\dim \calu \leq \tilde N$,
    \item $\calu$ is an $R$-long cover of  $\FS$,
    \item $\calu$ is $G$-invariant,
    \item $\calu$ has virtually cyclic isotropy:
      for all $U \in \calu$ the group 
      $G_U := \{ g \in G \mid g(U) = U \}$ is virtually cyclic.
  \end{enumerate}
  Let now $E := |\calu|$ be the nerve of the cover $\calu$. 
  The vertex set of this simplicial complex  is $\calu$ and we have 
  $|\calu| = \{ \sum_{U \in \calu} t_U U \mid t_U \in [0,1], \, 
           \sum_{U \in \calu} t_U = 1\;
                  \text{and} \; 
           \bigcap_{t_U \neq 0} U \neq \emptyset \}$.  
  Note that $|\calu|$ is a simplicial $(G,\VCyc)$-complex.
  To construct the desired map  $F \colon \FS \to E$ we first 
  change the metric on $\FS$.
  For (large) $\Lambda > 0$ we can define a metric that blows up
  the metric transversal to the flow $\phi$, and corresponds 
  to time along flow lines. 
  More precisely, 
  \begin{equation*}
   \begin{split}
    d_\Lambda (x,y) := \inf \Big\{ \sum_{i=1}^n \alpha_i + \Lambda\e_i \mid 
           \exists & x=x_0,x_1,\dots,x_n \; \text{such that} \\ 
           & d_\FS^\fol (x_{i-1},x_i) 
            \leq (\alpha_i,\e_i) \; \text{for} \; i=1,\dots,n \Big\}
   \end{split}
  \end{equation*}
  For $U \in \calu$, $x \in \FS$ let
  $a_{U} (x) := d_\Lambda ( x , \FS \setminus U )$
  and define $F \colon \FS \to |\calu|$ by
  \begin{equation*}
    F (x) := \sum_{U \in \calu} \frac{a_U(x)}{\sum_{V \in \calu} a_V(x)} U.
  \end{equation*}
  As $\calu$ is $G$-invariant, $F$ is $G$-equivariant.
  If $R > 0$ is sufficiently large (depending only on $\e$),  
  then there are $\Lambda > 0$ and $\delta > 0$ 
  (depending on everything at this point) such that 
  \begin{equation*}
    d_\FS^\fol ( x, x' ) \leq (\alpha, \delta) \; \implies
          d^1( F(x), F(x')) \leq \e. 
  \end{equation*}
  (More details for similar calculations can be found 
     in~\cite[Sec.~4.3, Prop.~5.3]{Bartels-Lueck-Reich(2008hyper)}.) 
  Thus we can compose with $F$ and conclude that 
  Proposition~\ref{prop:(alpha-e)-B-for-CAT(0)} 
  implies Proposition~\ref{prop:CAT(0)-satisfy-B}.
 
  \subsection*{The flow space for a $\CAT(0)$-space}
  
  This subsection contains an introduction to the flow space
  for $\CAT(0)$-groups from~\cite{Bartels-Lueck(2012CAT(0)flow)}.
  Let $Z$ be a $\CAT(0)$-space.
  
  \begin{definition}
    A \emph{generalized geodesic} in $Z$ is a continuous map 
    $c \colon \IR \to Z$ for which there exists an interval
    $(c_-, c_+)$ such that $c|_{(c_-,c_+)}$ is a geodesic and 
    $c|_{(-\infty,c_-)}$ and $c|_{(c_+,+\infty)}$ are constant. 
    (Here $c_- = -\infty$ and/or $c_+ = +\infty$ are allowed.)
  \end{definition}

  \begin{definition}
    The \emph{flow space for $Z$} is the space $\FS(Z)$ of all
    generalized geodesics $c \colon \IR \to Z$.
    It is equipped with the metric 
    \begin{equation*}
      d_\FS ( c, c' ) := \int_{\IR} \frac{d(c(t),c'(t))}{2 e^{|t|}}dt
    \end{equation*}
    and the flow 
    \begin{equation*}
      \phi_\tau( c) (t) := c( t + \tau ).
    \end{equation*}
  \end{definition}

  \begin{remark}
    The fixed point space for the flow  
    $\FS(Z)^\IR := \{ c \mid \phi_t(c) = c \; \text{for all} \; t \}$
    is via $c \mapsto c(0)$ canonically isometric to $Z$.

    The flow space $\FS(Z)$ is somewhat singular around $Z = \FS(Z)^\IR$.
    For example there are well defined maps $c \mapsto c(\pm\infty)$
    from $\FS(Z)$ to the 
    bordification~\cite[Ch.II.8]{Bridson-Haefliger(1999)} 
    $\bar Z$ of $Z$, but these maps fail to be continuous at $Z$.
  \end{remark}
  
  \begin{remark}
    The metric $d_\FS(c,c')$ cares most about $d(c(t),c'(t))$ for
    $t$ close to $0$.
    For example if $c(0) = c'(0)$  then $d_\FS(c,c')$ is bounded 
    by $\int_{0}^{\infty} \frac{t}{e^t} dt$.
    For this reason one can think of $c(0)$ as \emph{marking} 
    the generalized  geodesic $c$.
    If $c(0)$ is different from both $c(c_-)$ and $c(c_+)$
    (equivalently if $c_- < 0 < c_+$) then the triple
    $(c(c_-),c(0),c(c_+))$ uniquely determines $c$. 
  \end{remark}
  
  \begin{remark}
    An isometric action of $G$ on $Z$ induces an isometric action
    on $\FS(Z)$ via $(g \cdot c) (t) := g \cdot c(t)$.
    If the action of $G$ on $Z$ is in addition  cocompact, proper
    and $Z$ has  dimension at most $N$, then $\FS(Z)$ is a cocompact 
    $3N+2$-flow space for $G$ in the sense of 
    Definition~\ref{def:N-flow-space}, 
    see~\cite[Sec.~2]{Bartels-Lueck(2012CAT(0)flow)}.
 
    For cocompactness it is important that we allowed $c_- = -\infty$
    and $c_+ = +\infty$ in the definition of generalized geodesics. 
  \end{remark}

  \begin{remark}
    For hyperbolic groups there is a similar flow space constructed 
    by Mineyev~\cite{Mineyev(2005)}.
    This space is an essential ingredient for the
    proof that hyperbolic groups satisfy the assumptions of
    Theorem~\ref{thm:transfer-red-strict}.
    Mineyev's construction motivated the flow space for $\CAT(0)$ groups.

    However, for hyperbolic groups the construction is really much more 
    difficult. 
    A priori, there is really no local geometry associated to a 
    hyperbolic group, hyperbolicity is  just a condition on the large
    scale geometry and Mineyev extracts local information from this
    in the construction of his flow space.
    In contrast, for a $\CAT(0)$-group the corresponding $\CAT(0)$-space 
    provides local and global geometry right from the definition.
  \end{remark}

  \subsection*{Sketch of proof for 
         Proposition~\ref{prop:(alpha-e)-B-for-CAT(0)}}

  Let $Z$ be a finite dimensional $\CAT(0)$-space with an
  isometric, cocompact, proper action of the group $G$.
  Let $G = \langle S \mid R' \rangle$ be a finite presentation of $G$.
  Pick a base point $x_0 \in Z$.
  For $R > 0$ let $B_R(x_0)$ be the closed ball  in $Z$ of radius $R$ 
  around $x_0$.
  This will be our transfer space.
  Let $\rho_R \colon Z \to B_R(x_0)$ be the closest point projection.
  For $x, x' \in Z$, $t \in [0,1]$ we write 
  $t \mapsto (1-t) \cdot x + t \cdot x'$ for the straight line from
  $x$ to $x'$ parametrized by constant speed $d(x,x')$.
  For $g, h \in G$, $t \in [0,1]$, $x \in B_R(x_0)$ let
  \begin{eqnarray*}
    \varphi_g^R (x) & := & \rho_R ( g \cdot x), \\
    H^R_{g,h} (t,x) & := & \rho_R ( (1-t) \cdot g \varphi^R_h (x) + 
                                    t \cdot gh x ). 
  \end{eqnarray*}
  Then $H^R_{g,h}$ is a homotopy 
  $\varphi^R_g \circ \varphi^R_h \to \varphi^R_{gh}$.
  This data also specifies a homotopy action $(\varphi^R,H^R)$ 
  on $B_R(x_0)$. 
  We will use the map $\iota_R \colon B_R(x_0) \to \FS(Z)$ where
  $\iota_R(x)$ is the unique generalized geodesic $c$ in $Z$ with
  $c_- = 0$, $c_+ = d(x,x_0)$, $c(c_-) = c(0) = x_0$ and $c(c_+) = x$,
  i.e., the generalized geodesic from $x_0$ to $x$ that starts at time $0$
  at $x_0$.
  For $T \geq 0$ let
  $f^{T,R} := \phi_T \circ \iota_R \colon B_R(x_0) \to \FS(x_0)$.
  Proposition~\ref{prop:(alpha-e)-B-for-CAT(0)} follows from
  the next Lemma; this will conclude the sketch of proof for
  Proposition~\ref{prop:(alpha-e)-B-for-CAT(0)}.

  \begin{lemma} \label{lem:f^T,R...-does-it}
    Let $\alpha := \max_{s \in S} d(x_0, s x_0)$.
    For any $\e > 0$ there are $T, R > 0$ such that
    for all $x \in \FS$, $s \in S \cup S^{-1}$, $r \in R'$, $t \in [0,1]$
    we have
    \begin{itemize}
      \item  $d_\FS^\fol( f^{T,R}( \varphi^R_s (x) ), s \cdot f^{T,R}(x)) 
                 \leq (\alpha,\e)$,
      \item  $d_\FS^\fol ( f^{T,R}(H^{R}_r(t,x)), f^{T,R}(x))
                  \leq (\alpha,\e)$. 
    \end{itemize}
  \end{lemma}

  \begin{proof}[Sketch of proof]
    We will only discuss the first inequality; 
    the second inequality involves essentially no additional difficulties.
    
    Let us first visualize the generalized geodesics 
    $c := f^{T,R}( \varphi^R_s (x) )$ and $c' := s \cdot f^{T,R}(x)$.
    The generalized geodesic $c$ starts at $c(c_-) = x_0$ and ends at 
    $c(c_+) = \varphi^R_s(x)$. 
    If $d(x_0,sx) \leq R$, then the endpoint $\varphi^R_s(x)$ 
    coincides with $sx$; otherwise we can prolong  $c$ (as a geodesic)
    until it hits $sx$.
    If $T \leq d(x_0, \varphi^R_s (x)$ then $c(0)$ is the unique point on 
    the image of $c$  of distance $T$ from $x_0$, 
    otherwise $c(0) = c(c_+) = \varphi^R_s(x) = \rho_R(sx)$.
    The generalized geodesic $c'$ starts at $c'(c'_-) = s x_0$ and 
    ends at $c'(c'_-) = sx$.
    If $T \leq d(sx_0 ,sx)$, then $c'(0)$ is the unique point  on 
    the image of $c'$  of  distance $T$ from $s x_0$,
    otherwise $c'(0) = c'(c'_+) = sx$.
    We  draw this as 
    
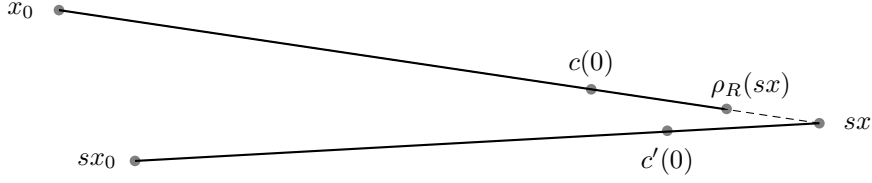
\begin{equation*}
     \begin{tikzpicture}
      \clip (-1,-2) rectangle (11.3,1.5);
      \fill [black,opacity=.5] (1,-1) circle (2pt);
      \draw (.5,-1) node {$s x_0$};
      \fill [black,opacity=.5] (0,1) circle (2pt);
      \draw (-.5,1) node {$x_0$};
      \fill [black,opacity=.5] (10,-.5) circle (2pt);
      \draw (10.5,-.5) node {$sx$};
      \fill [black,opacity=.5] (8.78,-.31) circle (2pt);
      \draw (9.1,0) node {$\rho_R(sx)$};
      \fill [black,opacity=.5] (8,-.6) circle (2pt);
      \draw (8,-1) node {$c'(0)$};
      \fill [black,opacity=.5] (7,-.05) circle (2pt);
      \draw (7,.3) node {$c(0)$};
      \draw[thin, densely dashed] (0,1) -- (10,-.5); 
      \draw[thick] (1,-1) -- (10,-.5); 
      \begin{scope}
        \clip (0,1) circle (8.85);
        \draw[thick] (0,1) -- (10,-.5);
      \end{scope}
     \end{tikzpicture}.
    \end{equation*}
    There are two basic cases to consider.\\[.7ex]
    \emph{Case I:} $d(sx,x_0)$ is small.\\
    Then $\rho_R(sx) = sx$, and both $c$ and $c'$ 
    converge to the constant geodesic at $sx$ with $T \to \infty$.
    Consequently $d_\FS(c,c')$ is small for large $T$.\\[.7ex]
    \emph{Case II:} $d(sx,x_0)$ is large.\\
    Then we may have $\rho_R(sx) \neq sx$. 
    Note that $d( \rho_R(sx), sx) \leq d( x_0, s x_0) \leq \alpha$.
    Let $t := d (c(0), sx) - d( c'(0), sx) \in [-\alpha, \alpha]$.
    Using  the $\CAT(0)$-condition one can then check that 
    $d_\FS ( \phi_t(c), c' )$  will be small provided that 
    $T$, $R - T$,  $\frac{R}{R-T}$ are large. 
  
    A more careful analysis of the two cases shows that
    it is possible to pick $R$ and $T$ (depending only on $\e$)
    such that for any $x$ one of the two cases applies and 
    therefore $d_\FS^\fol(c,c') \leq (\alpha,\e)$.
  \end{proof}

  \begin{remark} \label{rem:cocompact-is-important}
    The assumption that the action of $G$ on the $\CAT(0)$-space $Z$
    is cocompact is important for the proof of 
    Proposition~\ref{prop:CAT(0)-satisfy-B}, because it
    implies that the action of $G$ on the flow space $\FS(Z)$
    is also cocompact.
    This in turn is important for the construction of
    $R$-long covers: Theorem~\ref{thm:ex-long-thin-cover}
    otherwise only allows the construction of $R$-long covers
    for a cocompact subspace of the flow space.

    Nevertheless, there are situations where it is possible to 
    construct $R$-long 
    covers for flow spaces that are not cocompact.
    For example $\GL_n(\IZ)$ acts properly and isometrically but not
    cocompactly  on a $\CAT(0)$ space.
    But it is possible to construct $R$-long covers for the
    corresponding flow space~\cite{Bartels-Lueck-Reich-Rueping(2012)}.
    This uses as an additional input a construction of 
    Grayson~\cite{Grayson(1984)} and enforces a larger
    family of isotropy groups for the cover.
    This is the family $\calf_{n-1}$ mentioned in 
    Remark~\ref{rem:induction-GL_nZ}.   

    There are very general results of 
    Farrell-Jones~\cite{Farrell-Jones(1998)} 
    without a cocompactness assumption, but
    I have no good understanding of these methods.
  \end{remark}

  \section{$\IZ^n \rtimes \IZ$ as a Farrell-Hsiang group}
     \label{sec:Z^nxZ-is-Farrell-Hsiang}

  For $A \in \GL_n(\IZ)$ let $\IZ^n \rtimes_A \IZ$
  be the corresponding semi-direct product.
  We fix a generator $t \in \IZ$. 
  Then for $v \in \IZ^n$ we have $t \cdot v t^{-1} = Av$
  in $\IZ^n \rtimes_A \IZ$.
  The goal of this section is to outline a proof of the
  following fact from~\cite{Bartels-Farrell-Lueck(2011cocomlat)}.
  Recall that $\Ab$ denotes the family of abelian subgroups.
  In the case of $\IZ^n \rtimes_A \IZ$ these are all
  finitely generated free abelian. 

  \begin{proposition}
    \label{prop:IZ^nxIZ-is-Farrell-Hsiang}
    Suppose that no eigenvalue of $A$ over $\IC$ is a root of unity. 
    Then the group $\IZ^n \rtimes_A \IZ$ is a Farrell-Hsiang group
    with respect to the family $\Ab$ of abelian groups, i.e.,
    there is $N$  such that for any $\e > 0$
    there are
    \begin{enumerate}
    \item a group homomorphism $\pi \colon \IZ^n \rtimes_A \IZ \to F$ 
       where $F$ is finite,
    \item \label{prop:IZ^nxIZ-is-Farrell-Hsiang:complex}
      a simplicial $(\IZ^n \rtimes_A \IZ,\Ab)$-complex 
      $E$ of dimension at most $N$ 
    \item \label{prop:IZ^nxIZ-is-Farrell-Hsiang:map}
      a map $f \colon \coprod_{H \in \calh(F)} 
                \IZ^n \rtimes_A \IZ / \pi^{-1}(H) \to E$
      that is $\IZ^n \rtimes_A \IZ$-equivariant up to $\e$:
      $d^1(f(s x), s \cdot f(x)) \leq \e$ for all 
      $s \in S$, $x \in \coprod_{H \in \calh(F)} G/\pi^{-1}(H)$.  
    \end{enumerate}
    Here $S$ is any fixed generating set for $G$.
  \end{proposition}

  \begin{remark}
    The Farrell-Jones Conjecture holds for abelian groups.
    Thus using Theorem~\ref{thm:Farrell-Hsiang-groups} and the
    transitivity principle~\ref{prop:transitivity} we deduce from
    Proposition~\ref{prop:IZ^nxIZ-is-Farrell-Hsiang}
    that the Farrell-Jones Conjecture holds for the 
    group $\IZ^n \rtimes_A \IZ$ from 
    Proposition~\ref{prop:IZ^nxIZ-is-Farrell-Hsiang}.
  \end{remark}
  
  \subsection*{Finite quotients of $\IZ^n \rtimes_A \IZ$.}
  
  We write $\IZ / s$ for the quotient ring (and underlying cyclic group) 
  $\IZ / s \IZ$.
  Let $A_s$ denote the image of $A$ in $\GL_n( \IZ / s)$.
  Choose $r, s \in \IN$ such that the order $|A_s|$ of  $A_s$ 
  in $\GL_n( \IZ / s)$ divides $r$.
  Then we can form $(\IZ/s)^n \rtimes_{A_s} \IZ/r$ and there is 
  canonical surjective group homomorphism
  \begin{equation*}
    \pi \colon \IZ^n \rtimes_{A} \IZ \twoheadrightarrow  
         (\IZ/s)^n \rtimes_{A_s} \IZ/r.
  \end{equation*}

  \subsection*{Hyper-elementary subgroups of 
                                       $(\IZ/s)^n \rtimes_{A_s} \IZ/r$.}
  
  \begin{lemma}
    \label{lem:hyp-elm-subgr}
    Let $s = p_1 \cdot p_2$ be the product of two primes.
    Let $r := s \cdot |A_s|$.
    If $H$ is a hyper-elementary subgroup of 
    $(\IZ / s)^n \rtimes_{A_s} \IZ / r$
    then there is $q \in \{ p_1, p_2 \}$ such that
    \begin{enumerate}
    \item \label{lem:hyp-elm-subgr:Z^n}
        $\pi^{-1}(H) \cap \IZ^n \subseteq (q\IZ)^n$  
        or
    \item \label{lem:hyp-elm-subgr:Z} the image of $\pi^{-1}(H)$ under 
      $\IZ^n \rtimes_{A_s} \IZ \twoheadrightarrow \IZ$
      is contained in $q \IZ$. 
    \end{enumerate}
  \end{lemma}
  
  To prove Lemma~\ref{lem:hyp-elm-subgr} we 
  recall~\cite[Lem.~3.20]{Bartels-Farrell-Lueck(2011cocomlat)}.

  \begin{lemma}    \label{lem:prime_power}
    Let $s$ be any natural number.  
    Let $r := s \cdot |A_s|$.
    Let $C$ be a cyclic subgroup of ${\IZ/s}^n \rtimes_{A_s} \IZ/r$ 
    that has nontrivial intersection with $(\IZ/s)^n$.

    Then there is a prime power $q^N$ ($N \geq 1$) such that
    \begin{itemize}
     \item $q^N$ divides $r = r's$,
     \item $q^N$ does not divide the order of the image of $C$ in $\IZ/r$,
     \item $q$ divides the order of $C \cap (\IZ/s)^n$.
    \end{itemize}
  \end{lemma}

  \begin{proof}[Proof of Lemma~\ref{lem:hyp-elm-subgr}]
    Let $H \subseteq (\IZ/s)^n \rtimes_{A_s} \IZ/r$ be hyper-elementary.
    There is a short exact sequence 
    $C \rightarrowtail H \twoheadrightarrow P$ 
    with $P$ a $p$-group and $C$ a cyclic group.
    The cyclic group $C$ can always be arranged to be of order prime
    to $p$.
    \begin{equation*}
      \xymatrix{
         (\IZ / s)^n \ar@{>->}[rr] & &
         (\IZ/s)^n \rtimes_{A_s} \IZ/r \ar@{->>}[rr]^{\pr} & &
         \IZ / r \\
         H \cap (\IZ / s)^n \ar@{>->}[rr] \ar@{>->}[u] & &
         H \ar@{->>}[rr]^{\pr} \ar@{>->}[u] & &
         \pr(H) \ar@{>->}[u] \\
         C \cap (\IZ / s)^n \ar@{>->}[rr] \ar@{>->}[u] & &
         C \ar@{->>}[rr]^{\pr} \ar@{>->}[u] & &
         \pr(C)  \ar@{>->}[u]
      } 
    \end{equation*}
    There are two cases.\\[.7ex]
    \emph{Case I:} $C \cap (\IZ/s)^n$ is trivial.\\
    Then $H \cap (\IZ/s)^n$ is a $p$-group.
    Let $q$ be the prime from $\{ p_1, p_2 \}$ that
    is different from $p$.
    Then~\ref{lem:hyp-elm-subgr:Z^n} will hold. 
    \\[.7ex]
    \emph{Case II:} $C \cap (\IZ/s)^n$ is nontrivial.\\
    Then there is a prime $q$ as in Lemma~\ref{lem:prime_power}.
    As $q$ divides $| C \cap (\IZ / s)^n |$ we have $q \in \{p_1,p_2\}$
    and $q \neq p$.
    It follows that $q$ divides $[\IZ / r : \pr(H)]$. 
    This implies~\ref{lem:hyp-elm-subgr:Z}. 
  \end{proof}
  
  \subsection*{Contracting maps}
  
  \begin{example}
    \label{ex:contracting-to-R^n}
    Consider the standard action of $\IZ^n$ on $\IR^n$.
    Let $\bar H := (l\IZ)^n \subseteq \IZ^n$ and
    $\varphi \colon \bar H \to \IZ^n$ be the isomorphism
    $v \mapsto \frac{v}{l}$. 
    Let $\res_{\varphi} \IR^n$ be the $\bar H$-space obtained by
    restricting the action of $\IZ^n$ on $\IR^n$ with $\varphi_l$.
    Then $x \mapsto \frac{x}{l}$ defines an $\bar H$-map
    $F \colon \IZ^n \to \res_{\varphi} \IR^n$.
    This map is contracting.
    In fact by increasing $l$ we can make $F$ as contracting as
    we like, while we can keep the metric on $\IR^n$ fixed.
    
    A variant of this simple construction will be used to 
    produce maps as in~\ref{prop:IZ^nxIZ-is-Farrell-Hsiang:map}
    of Proposition~\ref{prop:IZ^nxIZ-is-Farrell-Hsiang}. 
    This will finish the discussion of  the proof of 
    Proposition~\ref{prop:IZ^nxIZ-is-Farrell-Hsiang}.
  \end{example}

  \begin{proposition}
    \label{prop:eps-maps-Z}
    Let $S \subseteq \IZ^n \rtimes_A \IZ$ be finite.
    For any $\e > 0$ there is $l_0$ such that for all 
    $l \geq l_0$ the following holds.

    Let $\bar H := \IZ^n \rtimes_{A} (l\IZ) 
                   \subseteq \IZ^n \rtimes_A \IZ$.
    Then there is a simplicial $(\IZ^n \rtimes_A \IZ,\Ab)$-complex 
    $E$ of dimension $1$ and an $\bar H$-equivariant map 
    \begin{equation*}
          F \colon \IZ^n \rtimes_A \IZ  \to E
    \end{equation*}
    such that $d^1( F(g), F(h) ) \leq \e$ whenever $g^{-1}h \in S$.
  \end{proposition}

  \begin{proof}
    We apply the construction of Example~\ref{ex:contracting-to-R^n}
    to the quotient $\IZ$ of $\IZ^n \rtimes_A \IZ$.
 
    Let $E := \IR$. 
    We use the standard way of making $E = \IR$ a simplicial complex
    in which $\IZ \subseteq \IR$ is the set of vertices.
    Let $\bar H$ act on $E$ via $(vt^k) \cdot \xi := \frac{k}{l} \xi$;
    this is a simplicial action.
    Finally define $F \colon \IZ^n \rtimes_A \IZ  \to E$ by
    $F(vt^k) := \frac{k}{l}$. 
    It is very easy to check that $F$ has the required properties 
    for sufficiently large $l$.
  \end{proof}

  \begin{proposition}
    \label{prop:eps-maps-Z^n}
    Let $S \subseteq \IZ^n \rtimes_A \IZ$ be finite.
    There is $N \in \IN$ depending only on $n$ such
    that for any $\e > 0$ there is $l_0$ such that for all
    $l \geq l_0$ the following holds.

    Let $\bar H :=  (l\IZ)^n \rtimes_A \IZ \subseteq \IZ^n \rtimes_A \IZ$.
    Then  there is   a simplicial $(\IZ^n \rtimes_A \IZ,\Cyc)$-complex 
    $E$ of dimension at most $N$ and an $\bar H$-equivariant map 
    \begin{equation*}
          F \colon \IZ^n \rtimes_A \IZ  \to E
    \end{equation*}
    such that $d^1( F(g), F(h) ) \leq \e$ whenever $g^{-1}h \in S$.
  \end{proposition}

  \begin{proof}[Sketch of proof]
    As in the proof of Proposition~\ref{prop:eps-maps-Z}
    we start with the construction from 
    Example~\ref{ex:contracting-to-R^n}, now
    applied to the subgroup $\IZ^n \subseteq \IZ^n \rtimes_A \IZ$.
    However, unlike the quotient $\IZ$, there is no homomorphism
    from  $\IZ^n \rtimes_A \IZ$ to the subgroup and it will be more
    difficult to finish the proof. 

    Let $\IZ^n \rtimes_A \IZ$ act on $\IR^n \x \IR$ via
    $vt^k \cdot (x,\xi) := (v + A^k(x),k + \xi)$.
    Let $\varphi \colon \bar H \to \IZ^n \rtimes_A \IZ$
    be the isomorphism $vt^k \mapsto \frac{v}{l}t^k$.
    The map $F_0 \colon \IZ^n \rtimes_A \IZ \to 
                              \res_{\varphi} \IR^n \x \IR$, 
    $(vt^k) \mapsto (v/l,k)$ is $\bar H$-equivariant and contracting 
    in the $\IZ^n$-direction, but not in the $\IZ$-direction.
    In order to produce a map that is also contracting in
    the $\IZ$-direction we use the flow methods from 
    Section~\ref{sec:flow-spaces}.

    There is $\IZ^n \rtimes_A \IZ$-equivariant flow on 
    $\IR^n \x \IR$ defined
    by $\phi_\tau(x,\xi) = (x,\tau + \xi)$. 
    It is possible 
    to produce a simplicial $(\IZ^n \rtimes_A \IZ,\Cyc)$-complex 
    $E$ of uniformly bounded dimension (depending only on $n$)
    and $\IZ^n \rtimes_A \IZ$-equivariant 
    map $F_1 \colon \IR^n \x \IR \to E$ 
    that is contracting in the  flow direction
    (but expanding in the transversal $\IR^n$-direction).
    To do so one uses Theorem~\ref{thm:ex-long-thin-cover};
    $E$ will be the nerve of a suitable long cover of $\IR^n \x \IR$.

    The fact that $F_1$ is expanding in the $\IR^n$-direction can be 
    countered by the contracting property of $F_0$.
    All together, the composition 
    $F_1 \circ F_0 \colon \IZ^n \rtimes_A \IZ \to \res_{\varphi} E$
    has the desired properties. 
  \end{proof}

  \begin{remark}
    As many other things, the idea of using a flow space
    in the proof of Proposition~\ref{prop:eps-maps-Z^n}
    originated in the work of Farrell and 
    Jones~\cite{Farrell-Jones(1988b)}.
    I found this trick very surprising when I first learned about it.
  \end{remark}

  \begin{lemma}
    \label{lem:barH}
    Let $\bar{H}$ be a subgroup of $\IZ^n \rtimes_A \IZ$
    and $l,k \in \IN$ such that 
    \begin{enumerate}
    \item \label{lem:barH:cap} $\bar{H} \cap \IZ^n \subseteq l\IZ$,
    \item \label{lem:barH:k} $\bar{H}$ maps to $k\IZ$ under the projection
       $\IZ^n \rtimes_A \IZ \to \IZ$,
    \item \label{lem:barH:lequiv1}
        the index $[\IZ^n : (\id - A^k)\IZ^n]$ is finite and 
        $l \equiv 1$ modulo $[\IZ^n : (\id - A^k)\IZ^n]$.
    \end{enumerate}
    Then $\bar{H}$ is subconjugated to $(l\IZ)^n  \rtimes_A \IZ$.
  \end{lemma}

  \begin{proof}
    Consider the image $\bar H_l$ of $\bar H$ under 
    $\IZ^n \rtimes_A \IZ \to (\IZ/l)^n \rtimes_A \IZ$.
    Then \ref{lem:barH:cap} implies that
    the restriction of the projection $(\IZ/l)^n \rtimes_A \IZ \to \IZ$
    to $\bar H_l$ is injective. 
    In particular $\bar H_l$ is cyclic.
    By \ref{lem:barH:k} there is $v \in \IZ^n$ such that 
    $vt^k \in \IZ^n \rtimes_A \IZ$ maps to a generator of $\bar H_l$.
    Assumption \ref{lem:barH:lequiv1} 
    implies that there is $w \in \IZ^n$ such that $v \equiv (\id - A^k)w$
    modulo $(l\IZ)^n$.
    A calculation shows that $w$ conjugates $\bar H$ to a subgroup of 
    $(l\IZ)^n  \rtimes_A \IZ$.
  \end{proof}

  \begin{proof}[Proof of Proposition~\ref{prop:IZ^nxIZ-is-Farrell-Hsiang}]
    Let $L$ be a large number.
    Since $A$ has no roots of unity as eigenvalues,
    the index $i_k := [\IZ^n : (\id - A^k) \IZ^n]$ is
    finite for all $k$.
    Let $K := i_1 \cdot i_2  \cdots  i_L$.
    By a theorem of Dirichlet there are infinitely many primes
    congruent to $1$ modulo $K$.
    Let $s = p_1 \cdot p_2$ be the product of two such primes,
    both $\geq L$,
    and set $r := s \cdot |A_s|$.
    
    We use the group homomorphism 
    $\pi \colon \IZ^n \rtimes_A \IZ \twoheadrightarrow  
         (\IZ/s)^n \rtimes_{A_s} \IZ/r$.
    Because of Lemma~\ref{lem:hyp-elm-subgr} we find for any
    hyper-elementary subgroup $H$ of $(\IZ/s)^n \rtimes_{A_s} \IZ/r$
    an $q \in \{p_1,p_2\}$ such that
    $\pi^{-1}(H) \subseteq \IZ^n \rtimes_A (q \IZ)$
    or $\pi^{-1}(H) \cap \IZ^n \subseteq (q \IZ)^n$.
    In the first case we set $l := q$.
    In the second case we have either
    $\pi^{-1}(H) \subseteq \IZ^n \rtimes_A (l \IZ)$   
    for some $l > L$ or
    we can apply Lemma~\ref{lem:barH} to deduce that (up to conjugation)
    $\pi^{-1}(H) \subseteq (q\IZ)^n \rtimes_A \IZ$
    and we again set $l := q$.

    Therefore it suffices to find  simplicial 
    $(\IZ^n \rtimes_A \IZ,\Ab)$-complexes $E_1$, $E_2$ whose 
    dimension is bounded by a number depending only on $n$ 
    (and not on $l$) and maps 
    \begin{eqnarray*}
      f_1 \colon & \IZ^n \rtimes_A \IZ / (l \IZ)^n \rtimes_A \IZ & \to E_1
      \\
      f_2 \colon & \IZ^n \rtimes_A \IZ / \IZ^n \rtimes_A (l\IZ) & \to E_2
    \end{eqnarray*}
    that are  $G$-equivariant up to $\e$.
    If $f \colon \IZ^n \rtimes_A \IZ \to E$ is the map from
    Proposition~\ref{prop:eps-maps-Z^n}, then we can set
    $E_1 := (\IZ^n \rtimes_A \IZ) \x_{(l \IZ^n) \rtimes_A \IZ} E$
    and define $f_1$ by $f(vt^k) :=  \big( (vt^k) , f((vt^k)^{-1}) \big)$.
    Similarly, we can produce $f_2$ using
    Proposition~\ref{prop:eps-maps-Z}.
  \end{proof}



\def\cprime{$'$} \def\polhk#1{\setbox0=\hbox{#1}{\ooalign{\hidewidth
  \lower1.5ex\hbox{`}\hidewidth\crcr\unhbox0}}}

\end{document}